\documentclass[11pt]{amsart}

\usepackage{amssymb} 
\usepackage{amsmath}
\usepackage{amscd}

\newtheorem{thm}{Theorem}[section]

\newtheorem{lem-dfn}[thm]{Lemma-Definition}
\newtheorem{prop}[thm]{Proposition}
\newtheorem{cor}[thm]{Corollary}

\theoremstyle{definition}
\newtheorem{defn}[thm]{Definition}
\newtheorem{exam}[thm]{Example}
\newtheorem{ex}[thm]{Example}

\newtheorem{quest}[thm]{Question}

\newtheorem*{acknowledgement}{Acknowledgement}
\theoremstyle{remark}
\newtheorem{clm}{Claim}
\newtheorem{rem}[thm]{Remark}
\numberwithin{equation}{section}
\newcommand{\cF}{\mathcal F}
\newcommand{\cI}{\mathcal I}
\newcommand{\cL}{\mathcal L}
\newcommand{\cO}{\mathcal O}
\newcommand{\cR}{\mathcal R}
\newcommand{\m}{\mathfrak m}
\newcommand{\n}{\mathfrak n}
\newcommand{\MM}{\mathfrak M}
\newcommand{\bbN}{\mathbb N}
\newcommand{\bbZ}{\mathbb Z}

\DeclareMathOperator{\Spec}{Spec}

\DeclareMathOperator{\Proj}{Proj}

\DeclareMathOperator{\supp}{Supp}
\DeclareMathOperator{\nr}{nr}
\DeclareMathOperator{\br}{\bar r}
\DeclareMathOperator{\chara}{char}

\DeclareMathOperator{\di}{div}

\DeclareMathOperator{\Aut}{Aut}

\renewcommand{\:}{\colon}
\begin{document}

\title[Normal Hilbert coefficients and  elliptic ideals]{Normal Hilbert coefficients and  elliptic ideals in normal two-dimensional singularities}

%\author{
%  \name[T.Okuma]{Tomohiro Okuma}
%  \authorfootnote{Partially supported by JSPS Grant-in-Aid 
%for Scientific Research (C) Grant Number 17K05216.}
%  \address{Department of Mathematical Sciences, 
%Yamagata University,  Yamagata, 990-8560, Japan}
%  \email{okuma@sci.kj.yamagata-u.ac.jp}
%}
%
%\author{
%  \name[M.E.Rossi]{Maria Evelina Rossi}
%  \authorfootnote{Partially supported by PRIN 2020355B8Y.}
%  \address{Dipartimento di Matematica
%Universita' degli Studi di Genova
%Via Dodecaneso 35 I-16146 Genova}
%  \email{rossim@dima.unige.it}
%}
%
%\author{
%  \name[K.-i.Watanabe]{Kei-ichi Watanabe}
%  \authorfootnote{Partially supported by 
%  JSPS Grant-in-Aid for Scientific Research (C) Grant Number 
%  20K03522 and by GNSAGA INdAM, Italy.}
%  \address{Department of Mathematics, College of Humanities and 
%  Sciences, Nihon University, Setagaya-ku, Tokyo, 156-8550, Japan 
%  and Organization for the Strategic Coordination of Research and 
%  Intellectual Properties, Meiji University}
%  \email{ watnbkei@gmail.com}
%}
%
%\author{
%  \name[K. Yoshida]{Ken-ichi Yoshida}
%  \authorfootnote{Partially supported by JSPS Grant-in-Aid 
%for Scientific Research (C) Grant Number 19K03430.}
%  \address{Department of Mathematics, 
%College of Humanities and Sciences, 
%Nihon University, Setagaya-ku, Tokyo, 156-8550, Japan}
%  \email{yoshida.kennichi@nihon-u.ac.jp}
%}

\author[T.Okuma]{Tomohiro Okuma}
  \address{Department of Mathematical Sciences, 
Yamagata University,  Yamagata, 990-8560, Japan}
  \email{okuma@sci.kj.yamagata-u.ac.jp}

\author[M.E.Rossi]{Maria Evelina Rossi}
  \address{Dipartimento di Matematica
Universita' degli Studi di Genova
Via Dodecaneso 35 I-16146 Genova}
  \email{rossim@dima.unige.it}

\author[K.-i.Watanabe]{Kei-ichi Watanabe}
  \address{Department of Mathematics, College of Humanities and 
  Sciences, Nihon University, Setagaya-ku, Tokyo, 156-8550, Japan 
  and Organization for the Strategic Coordination of Research and 
  Intellectual Properties, Meiji University}
  \email{ watnbkei@gmail.com}

\author[K. Yoshida]{Ken-ichi Yoshida}
  \address{Department of Mathematics, 
College of Humanities and Sciences, 
Nihon University, Setagaya-ku, Tokyo, 156-8550, Japan}
  \email{yoshida.kennichi@nihon-u.ac.jp}


\thanks{
The first author was partially supported by JSPS Grant-in-Aid 
for Scientific Research (C) Grant Number 17K05216.
The second author was partially supported by PRIN 2020355B8Y.
The third was partially supported by 
  JSPS Grant-in-Aid for Scientific Research (C) Grant Number 
  20K03522 and by GNSAGA INdAM, Italy.
The fourth author was partially supported by JSPS Grant-in-Aid 
for Scientific Research (C) Grant Number 19K03430.
}

\maketitle

\begin{abstract}
Let $(A,\m)$ be an excellent  two-dimensional normal local domain. 
In this paper we study the elliptic and the strongly elliptic ideals  of $A$ with the aim to characterize  elliptic and strongly elliptic singularities,   according to the definitions given  by  Wagreich and by Yau.  
In analogy with the rational singularities, in the main result we characterize a strongly elliptic singularity in terms  of the normal Hilbert coefficients  of the integrally closed $\m$-primary ideals of $A$.    
Unlike  $p_g$-ideals, elliptic  ideals and  strongly elliptic ideals are not necessarily normal and   necessary and sufficient conditions for being normal are given. 
In the last section we discuss the existence (and the effective construction) of strongly elliptic ideals in any two-dimensional normal local ring.
\end{abstract}
%%%%%%%%%%%%%%%%%%%%%%%%%%%%%%%%%%%%%%%%%%%%%%%%%%%%%%%
\section{Introduction and Notations}
Let $(A,\m)$ be an excellent  two-dimensional normal local ring 
and let $I$ be an $\m$-primary ideal of $A$.  
The integral closure $\bar{I}$ of $I$ is the ideal consisting of all solutions $z$ of some equation with coefficients $c_i \in I^i $:  
$Z ^n +c_1 Z^{n-1} +c_2 Z^{n-2} + \dots+ c_{n-1} Z + c_n=0$. 
Then $ I \subseteq \bar{I} \subseteq \sqrt{I}$. 
We say that $I$ is \textit{integrally closed} if $I = \bar{I}$ and $I$ is \textit{normal} if $I^n= {\overline{I^n}} $ for every positive integer $n$.  
By a classical result of Rees \cite{R}, under our assumptions, 
the filtration $\{\overline{I^n}\}_{n \in \bbN}$ is a good $I$-filtration of $A$ 
and it is called the \textit{normal filtration}. 
%%%%%%%%%%%%%%%%
\par \vspace{1mm}
We may define the Hilbert Samuel function $\bar{H}_I(n):=\ell_A(A/\overline{I^{n+1}})$ for all integers $n \ge 0$ and it becomes a polynomial for large $n$   
(here $\ell_A(M) $ is the length of the $A$-module $M$). 
This polynomial  is called the \textit{normal Hilbert polynomial}
\[
\bar{P}_I(n) =\bar e_0(I) \binom{n+2}{2}
%\genfrac{(}{)}{0pt}{0}{n+2}{2}
-\bar e_1(I) \binom{n+1}{1}
%\genfrac{(}{)}{0pt}{0}{n+1}{1}
+\bar e_2(I), 
\]
and the  coefficients  $\overline{e}_i(I)$, $i=0,1,2,$ are the 
\textit{normal Hilbert coefficients}. 

\par 
A rich literature is available  on the normal Hilbert coefficients $\bar e_i(I)$ and this study is considered an important part of the theory of blowing-up rings, see for instance  
\cite{CPR1, CPR2, Hun, It1, It2, MOR, MORT, MSV, RV}. 
%\cite{Hun},  \cite{It1}, \cite{It2}, \cite{RV},  \cite{CPR1}, \cite{CPR2},  \cite{MSV}, \cite{MOR}, \cite{MORT}.   
\par 
From the geometric side, any   integrally closed  $\m$-primary 
ideal $I$ of $A  $ is represented on some resolution, see \cite{Li}. 
Let $$f \colon X \to \Spec A$$  be a resolution of singularities  
with  an anti-nef cycle $Z>0$ on $X$ so that 
$I =I_Z= H^0(\mathcal{O}_X(-Z))$ 
 and $I\mathcal{O}_X=\mathcal{O}_X(-Z)$.    
We say that $I = I_Z$ is \textit{represented by $Z $ on $X$}.    
The aim of this paper is to join the algebraic 
and the geometric information on $A$ 
taking advantage of the theory of the Hilbert functions and 
of  the theory of the resolution of singularities. 
\par 
For  a coherent $\mathcal{O}_X$-module $\mathcal{F}$, we write 
$h^i(\mathcal{F})=\ell_A(H^i(X, \mathcal{F}))$.   
If $I=I_Z$ is an $\m$-primary integrally closed ideal of $A$ 
represented by $Z$ on $X$,  one can define for every integer 
$n \ge 0$  a decreasing chain of integers 
$q(nI):= q(\overline{I^n}) = h^1(\mathcal{O}_X(-nZ))$ 
where $q(0I) := p_g(A)$ is the  \textit{geometric genus}  of $A$.       
It is proved  that  $q(nI) $ stabilises for every $I$ and $n\ge p_g(A)$. 
We denote it by $q(\infty I)$. 
\par 
These integers are independent of the representation 
and they are strictly related to the  normal Hilbert polynomial.  
The key of our approach can be considered Theorem \ref{kato} and  
Proposition \ref{p:normalHP},  consequences of Kato's Riemann-Roch 
formula (see \cite{kato} and \cite{OWY2}).   
In particular the following holds:
\begin{enumerate}
\item $\overline{P}_I(n)=\ell_A(A/\overline{I^{n+1}})$ for all $n \ge p_g(A)-1$. 
%\vspace{1mm}
\item 
$\bar e_1(I)-e_0(I) + \ell_A(A/I) =p_g(A) - q(I)$. 
%\vspace{1mm}
\item $\bar e_2(I)= p_g(A)-q(nI)=p_g(A)-q(\infty I)$ for all $n \ge p_g(A)$. 
\end{enumerate}
Moreover, we have 
\[
\bar e_0(I)= -Z^2,\qquad \bar e_1(I)= \dfrac{-Z^2+ZK_X}{2}. 
\]
\par 
This makes the bridge between the theory of the normal Hilbert 
coefficients and the theory of the singularities. 
This is the line already traced by Lipman \cite{Li}, Cutkosky \cite{C} 
and more recently  by  Okuma,  Watanabe and  Yoshida, 
see \cite{OWY1, OWY2, OWY3}.  
\par 
Let $(A,\m)$ be a two-dimensional excellent normal local domain 
containing an algebraically closed field $k= A/\m$.   
It is known that $A$ is a rational singularity (see \cite{Ar}) if and only if 
every  integrally closed $\m$-primary ideal $I$ of $A$ is normal 
(see \cite{Li} and  \cite{C}), equivalently $\bar e_2(I)=0$,  that is  
$I$ is a $p_g$-ideal,  as proved in \cite{OWY1, OWY2}.  
Inspired by a paper by the first author \cite{Ok},  we  investigate the 
integrally closed $\m$-primary ideals  of \textit{elliptic singularities} 
(see   Wagreich \cite{W})  and of \textit{strongly elliptic singularities} 
(see  Yau  \cite{Y}).  
All the preliminary results are contained in Section 2. 

\vskip 1mm
In Section 3 we prove the main results of the paper. 
We define {\it{the elliptic and the strongly elliptic ideals}}  
aimed by the  study of non rational singularities. 
We recall that if  $Q$ is a minimal reduction of $I$, then 
we  denote by  
$\br(I):= \min\{ r \;|\; \overline{I^{n+1}} = Q \overline{I^n}\;
\text{for all} \; n \ge r\}$ \textit{the normal reduction number}  of $I$ 
and  this integer  exists and does not depend on the choice of $Q$. 
Okuma  proved that if $A$ is an elliptic singularity, then 
$\br(I) = 2$ for any integrally closed $\m$-primary ideal of $A$,  
see \cite[Theorem 3.3]{Ok}.   
According to Okuma's result,  we define  \textit{elliptic ideals} 
to be the integrally closed $\m$-primary ideals satisfying $\br(I)= 2$.  
In Theorem \ref{nr2} we prove that elliptic ideals satisfy   
$\bar e_2(I)=\bar{e}_1(I) -e_0(I) + \ell_A( A/I)>0 $ attaining  the  minimal 
value according to the inequality proved by Sally \cite{S2} and Itoh 
\cite{It2}. 
In particular if $I$ is an elliptic ideal,  then $p_g(A) > q(I)=q(\infty I)$. 
If $A$ is not a rational singularity, then elliptic ideals always exist, see Proposition \ref{exist}.  In particular  we prove
\vskip 1mm
\noindent {\bf{Proposition.}} (See Proposition 3.3.)   
{\it{If  $A$ is not a rational singularity, then 
for any $\m$-primary integrally closed ideal $I$ of $A$, 
$\overline{I^{n}}$ is either a $p_g$-ideal or 
an elliptic ideal for every $n \ge p_g(A)$. }}

\par \vspace{2mm}
Yau in \cite{Y}, Laufer  in  \cite{La} and  Wagreich in \cite{W} 
introduced  interesting classes of elliptic singularities.   
An excellent  two-dimensional normal local ring  
$A$ is a {{strongly  elliptic singularity}} if $ p_g(A)=1$,   
that is $p_g$ is almost minimal.    
\par 
Among the elliptic ideals, in Theorem \ref{e_2=1} we define  strongly elliptic ideals those for which $\bar e_2 =1 $ and equivalent conditions are given. 
The following result characterizes algebraically  the strongly  elliptic singularities. 
\vskip 1mm
%%%%%  Theorem 3.14
\noindent {\bf{Theorem.}} (See Theorem 3.14.) 
\textit{Let $(A,\m)$ be a two-dimensional excellent normal local domain containing an algebraically closed field $k= A/\m$ 
 and assume that $p_g(A)>0$. 
The following conditions are equivalent:}
\begin{enumerate}
\item \textit{$A$ is a strongly elliptic singularity.} 
\item \textit{Every integrally closed ideal of $A$ is either a $p_g$-ideal or  a  strongly  elliptic ideal.}
\end{enumerate}
\par 
Notice that $p_g$-ideals are always normal, but  
elliptic ideals are not necessary normal, see Proposition \ref{normal}, Examples \ref{exnormal} and \ref{noti}.  
Moreover if $A$ is strongly  elliptic and $I$  is not a $p_g$-ideal,   then 
Proposition \ref{normal} and Theorem \ref{I2} give necessary and sufficient 
conditions for being $I$ normal.
\vskip 1mm
\noindent {\bf{Theorem.}}  {\it{ Let $(A,\m)$ be a two-dimensional 
excellent normal local domain 
containing an algebraically closed field $k= A/\m$. 
Assume that $A$ is a strongly elliptic singularity. 
If $I = I_Z $ is an elliptic ideal   
$($equivalently $I$ is not a $p_g$-ideal$)$ and $D$ is 
the minimally elliptic cycle on $X$, then $I^2$  is integrally closed 
$($equivalently $I$ is normal$)$ if and only if $- Z D \ge 3$  and 
if $- ZD \le 2$,   
then $I^2 = QI$.    
}}
\vskip 1mm
For any normal surface singularity which is not rational, $p_g$-ideals and 
elliptic ideals exist plentifully. 
But this is no longer true for strongly elliptic ideals. 
\par \vskip 1mm
In Section 4, we show  that there exist  excellent  two-dimensional normal 
local rings  having  no strongly  elliptic ideals, see Examples \ref{no}. 
Finally, Corollary \ref{existence}  gives necessary and sufficient conditions 
for the existence of strongly elliptic ideals in terms of the existence of 
certain cohomological cycles. 
When there exist, we present an effective geometric construction, see 
Example \ref{si}. 

%%%%%%%%%%%%%%%%%%%%%%%%%%%%%%%%%%%%%%%%%%%%%%%%%%%%%%%%%
%%  Section 2
\section{Preliminaries and  normal reduction number}
 
Let $(A,\m)$ be an excellent  two-dimensional normal local domain  
containing an algebraically closed field $k= A/\m$  and let $I$ be an 
integrally closed  $\m$-primary ideal of $A$.  
With the already introduced  notation, then there exists a resolution 
$X \to \ \Spec A$ and a cycle $Z$ such that $I$ is represented on $X$ by 
$Z$. 
When we write $I_Z$ we always assume that $\mathcal{O}_X(-Z) $ is generated by 
global sections, namely $I \mathcal{O}_X=\mathcal{O}_X(-Z)$, and note that  
$I_Z= H^0(X, \mathcal{O}_X(-Z))$.  
Recall that the geometric genus $p_g(A)=  h^1(\mathcal{O}_X)$ is 
independent of the choice of the resolution. 
\par 
Okuma, Watanabe and Yoshida  introduced a natural extension of the 
integrally closed ideals in a two-dimensional rational singularity, 
that is the $p_g$-ideals. With the previous notation
\[
p_g(A) \ge h^1(\mathcal{O}_X(-Z) )
\]
and if the equality holds, then $Z$ is called a $p_g$-cycle and $I=I_Z$ is 
called a $p_g$-ideal. 
In \cite{OWY1, OWY2}, the authors 
characterized the $p_g$-ideals in terms of the normal Hilbert polynomial.  
They proved  that  $A$ is a rational singularity if and only if every integrally 
closed $\m$-primary ideal is a $p_g$-ideal. 
Starting by $p_g(A)$ we define the following chain of integers. 

%%%  Definition 2.1
\begin{defn} \label{qI} 
We define $q(I):= h^1(\mathcal{O}_X(-Z))$ and more in general 
$q(nI):= q(\overline{I^n}) = h^1(\mathcal{O}_X(-nZ))$ for every 
integer $n \ge 1$. 
\end{defn}
\par 
We put $q(0 I)=h^1(\mathcal{O}_X)=p_g(A)$.   
Notice that $q(nI)$  is in general very difficult to compute, but it is independent of the representation \cite[Lemma 3.4]{OWY1}. 
These invariants are strictly related to the normal Hilbert polynomial and their interplay will be very important in our approach.
\par  
The following formula is called a Riemann-Roch formula.  
The result was proved in \cite{kato} in the complex case, 
but it holds in any characteristic, see \cite{WY}. 

%%%%% Theorem 2.2
%\begin{thm}[\textbf{Kato's Riemann-Roch formula}]
\begin{thm}[\textbf{Kato's Riemann-Roch formula  \cite[Theorem 2.2]{WY}}] 
\label{kato}
Let $I=I_Z$ be an $\m$-primary integrally closed ideal represented 
by an anti-nef cycle $Z$ on $X$. 
Then we have 
\[
\ell_A(A/I) + q(I) =-\dfrac{Z^2+K_XZ}{2} + p_g(A),
\]
where $K_X$ denotes the canonical divisor. 
\end{thm}

\par
We recall here the properties of the sequence $\{q(nI)\}$, 
The following Propositions \ref{q(nI)} and \ref{eq:q(nI)} 
on $I = I_Z$ follows from the long exact sequence 
attached to the short exact sequence 
\[
(\dagger) \quad 0\to \mathcal{O}_X(-(n-1)Z) \to \mathcal{O}_X(-nZ)^{\oplus 2}\to \mathcal{O}_X(-(n+1)Z)\to 0.
\]
see \cite[Lemma 3.1]{OWY2}.   

%%%  Proposition 2.3
\begin{prop} \label{q(nI)}
With the previous notation, the following facts hold:  
\begin{enumerate}
\item $0 \le q(I) \le p_g(A);$ and 
\item $q(kI) \ge q((k+1)I)$ for every integer $k \ge 0$ and 
if $q(nI) = q((n+1)I)$ for some $n\ge 0$, 
then $q(nI) = q(mI)$ for every $m\ge n$. 
Hence $q(nI) = q((n+1)I)$ for every $I$ and $n\ge p_g(A)$. 
We denote it by $q(\infty I)$. 
\end{enumerate}
\end{prop}

\par 
We will use the above sequence  for computing the following important 
algebraic numerical  invariants of the normal filtration 
$\{\overline{I^n}\}$.   
Let $\mathbb{Z}_{+}$ denote the set of positive integers. 

%%%  Definition 2.4
\begin{defn}[\textrm{cf. \cite{OWY4}}] \label{nrns}
Let $I \subset A$ be $\m$-primary integrally closed ideal, 
and let $Q$ be a minimal reduction of $I$. 
Define:
\begin{eqnarray*}
\nr(I) &:=& \min\{r \in \mathbb{Z}_{+} \,|\, \overline{I^{r+1}}=Q\overline{I^r}\}, \\
\br(I) &:=& \min\{ r \in \mathbb{Z}_{+} \;|\; \overline{I^{n+1}} = Q \overline{I^n} \; \mbox{\rm for all } n \ge r\}.
\end{eqnarray*}
We call $\br(I)$ the {\it{normal reduction number}}  and 
$\nr(I)$ the \textit{relative normal reduction number}. 
\end{defn}

\par 
The normal reduction number   exists (see \cite{NR} and \cite{R}) and it 
has been studied by many authors in the context of the Hilbert function 
and of the Hilbert polynomial (e.g. see \cite{Hun}, \cite{It2}, \cite{CPR1}, 
\cite{CPR2}, \cite{MSV}). 
The main difficulty  of the normal filtration  with respect to the $I$-adic 
filtration, is that the Rees algebra of the normal filtration is not generated 
by the part of degree one because $I \overline{I^n} \neq  
\overline{I^{n+1}}$.   By the definition, we deduce that $\nr(I) \le \br(I)$ 
and we will see that in general they do not coincide. 
Note that the definitions of $\nr(I)$ and of $\br(I)$ are independent on the 
choice of  a minimal reduction $Q$ of $I$ 
(see e.g. \cite[Theorem 4.5]{Hun}).  
It is also a consequence of the following result  in \cite{OWY4},  \S2.

%%% Proposition 2.5
\begin{prop}\label{eq:q(nI)}
The following statements hold. 
\begin{enumerate}
\item 
For any integer $n \ge 1$, we have 
\[
2 \cdot q(nI) + \ell_A(\overline{I^{n+1}}/Q\overline{I^n})
=q((n+1)I)+q((n-1)I).  
\]
 \item We have
\begin{eqnarray*}
\nr(I) &=& \min\{n \in \mathbb{Z}_{+} \,| \, 
   q((n-1)I)- q(nI) = q(nI) - q((n+1)I)  \},\\
\br(I) &=& \min\{n \in \mathbb{Z}_{+} \,|\, q((n-1)I)=q(nI)  \}.
\end{eqnarray*}
\end{enumerate}
\end{prop}

\par 
From the propositions above we have that $\br(I) \le p_g(A) +1$.
In \cite[Theorem 2.9]{OWY4}  the authors showed that 
$p_g(A) \ge \binom{\nr(I)}{2}$.  

\par 
The second author \cite[Corollary 1.5]{Ro} proved the following  upper 
bound on the  reduction number $r(I)$ for every $\m$-primary ideal $I$ 
(here $r(I)$ denotes the reduction number for the $I$-adic filtration) in a 
two-dimensional Cohen-Macaulay local ring $A$ in terms of the Hilbert 
coefficients: 
\[
r(I) \le e_1(I) -e_0(I) + \ell_A(A/I)+1. 
\] 
The bound gives, as a consequence,  several interesting results, 
in particular a positive answer to a longstanding conjecture stated 
by J. Sally in the case of local Cohen-Macaulay rings of 
almost minimal multiplicity, see  \cite{Ro} and \cite{S1}.  
Later,  the inequality was extended by Rossi and Valla, see \cite{RV},  
Theorem $4.3$  to special multiplicative $I$-filtrations. 
The result does not include the normal filtration. 
It is natural to ask if  the same bound also holds for $\br(I)$.   
The answer is negative as we will show later, but we prove that the 
analogue  upper bound holds true for $\nr(I)$.  
We need  some preliminary results. 

\par \vspace{1mm}
From Riemann-Roch formula (Theorem $\ref{kato}$), we get 
\[
\ell_A(A/\overline{I^{n+1}})+q((n+1)I)
=-\dfrac{(n+1)^2Z^2+(n+1)ZK_X}{2} + p_g(A). 
\]
Using this, 
we can express $\bar{e}_0(I), \bar{e}_1(I), \bar{e}_2(I)$ as follows. 
%
%%% Proposition 2.6
\begin{prop}[\textrm{\cite[Theorem 3.2]{OWY2}}]   \label{p:normalHP}
Assume that $I = I_Z$ is represented by a cycle $Z > 0$
 on a resolution $X$ of  $\Spec(A)$. 
Let $\bar{P}_I(n)$ be the normal Hilbert-polynomial of $I$. 
Then 
\begin{enumerate}
\item $\overline{P}_I(n)=\ell_A(A/\overline{I^{n+1}})$ for all 
$n \ge p_g(A)-1$. 
%\vspace{1mm}
\item $\bar e_0(I)=e_0(I)$.
%\vspace{1mm}
\item $\bar e_1(I)-e_0(I) + \ell_A(A/I) =p_g(A) - q(I)$. 
%\vspace{1mm}
\item $\bar e_2(I)= p_g(A)-q(nI)=p_g(A)-q(\infty I)$ for all $n \ge p_g(A)$. 
\end{enumerate}
Moreover, we have 
\[
\bar e_0(I)= -Z^2,\qquad \bar e_1(I)= \dfrac{-Z^2+ZK_X}{2}. 
\]
\end{prop}

%%%%  Theorem 2.7
\begin{thm} \label{nr(I)}
Let $(A,\m)$ be an excellent two-dimensional 
normal local domain containing an algebraically closed field $k= A/\m$. 
Let $I \subset A$ be an $\m$-primary integrally closed ideal. 
Then 
\[
\nr(I) \le \bar{e}_1(I) -  \bar{e}_0(I) +\ell_A(A/I) +1. 
\]
If we put $r=\nr(I)$, 
equality holds if and only if 
the following conditions  hold true $:$
\begin{enumerate}
\item $\ell_A(\overline{I^{n+1}}/Q\overline{I^n}) =1$
for $n= 1,\ldots, r-1$ if $r>1,$ 
\item $q((r-1)I)=q(\infty I)$.
\end{enumerate}
When this is the case, $\nr(I)=\br(I)$, $q(I)=p_g(A)-\br(I)+1$, 
and $\bar{e}_2(I)=p_g(A)- q(\infty I) = r(r-1)/2$. 
\end{thm}

\begin{proof} 
By virtue of Proposition \ref{p:normalHP}, 
it is enough to show 
\[
\nr(I) \le p_g(A)-q(I)+1. 
\]
If we put $\Delta q(n) : =  q(nI) - q((n+1)I)$ for every integer $n \ge 0$, 
then $\Delta q(n) $ is non-negative and decreasing 
since $\ell_A(\overline{I^{n+1}}/Q\overline{I^n})
=\Delta q(n-1) -\Delta q(n)$.
We have 
\[
\nr(I)=\min\{n \in \mathbb{Z}_{+} \,|\, \Delta q(n-1) =\Delta q(n) \}, 
\quad
\br(I)=\min\{n \in \mathbb{Z}_{+} \,|\, \Delta q(n-1) =0 \}. 
\]
Put $a=p_g(A)-q(I)$. 
Then $\Delta q(0)=a\ge \nr(I)-1$ and $\nr(I)=a+1$ if and only if 
\[
\Delta q(0) =a > \Delta q(1) =a-1 > \cdots > \Delta q(a-1)=1 > 
\Delta q(a)=0 = \Delta q(a+1). 
\]
\par 
Now assume $\nr(I)=a+1$. Then $a=r-1$ and for every $n$ with 
$1 \le n \le a=r-1$, we have 
\[
\ell_A(\overline{I^{n+1}}/Q\overline{I^n})=\Delta q(n-1)-\Delta q(n)=
(a-n+1)-(a-n)=1. 
\]
Moreover, for every $n \ge a+1$, we have 
\[
\ell_A(\overline{I^{n+1}}/Q\overline{I^n})=\Delta q(n-1)-\Delta q(n)=0
\]
and thus $\overline{I^{n+1}}=Q\overline{I^n}$. 
Hence $\br(I)=a+1=\nr(I)$. 
Furthermore, 
\[
\bar{e}_2(I)=p_g(A)-q(\infty I)=q(0 I)-q((r-1)I) = \sum_{i=0}^{r-2} \Delta q(i)= 
\dfrac{r(r-1)}{2}. 
\] 
One can prove the converse similarly. \qed
\end{proof}

\par \vspace{1mm}
Note  that, if the equality holds in the previous result, then the normal 
filtration $\{\overline{I^n}\} $ has almost minimal multiplicity 
following the definition given in \cite[2.1]{RV}. 
In the following example we show that  
Theorem \ref{nr(I)}  does \textit{not} hold if we replace $\nr(I)$ 
by $\bar{r}(I)$.    
The example shows  that for all $g \ge 2$, 
there exist an excellent  two-dimensional normal local ring $A$ 
and an integrally closed $\m$-primary ideal $I$ 
such that $\nr(I) =1, \bar{r}(I)={{g+1}}, q(I) = g-1$ and 
$\ell_A(A/I) = g$. 

\par
The following ideal $I$ satisfies $\bar{e}_1(I)=\bar{e}_0(I)-\ell_A(A/I)+1$, 
but $\br(I)\not \le 2$.  

%%%%%%  Example 2.8
\begin{ex}[\textrm{\cite[Example 3.10]{OWY5}}] \label{e}
Let $g \ge 2$ be an integer, and let $K$ be 
 a field of $\chara K=0$ or $\chara K=p$, where 
 $p$ does not divide $2g+2$. 
Then $R=K[X,Y,Z]/(X^2-Y^{2g+2}-Z^{2g+2})$ is 
a graded normal $K$-algebra with $\deg X=g+1$, $\deg Y=\deg Z=1$. 
Let $A=R^{(g)}$ be the $g^{th}$ Veronese subring of $R$:
\[
A=K[y^{g},y^{g-1}z,y^{g-2}z^{2},\ldots, z^{g}, xy^{g-1},xy^{g-2}z,
\ldots,xz^{g-1}],
\]
where $x, y, z$ denotes, respectively, the image of $X, Y, Z$ in $R$.   
Then $A$ is a graded normal domain with $A_k=R_{kg}$ 
for every integer $k \ge 0$. 
Let $I=(y^g,y^{g-1}z) +A_{\ge 2}$ and $Q=(y^g-z^{2g}, y^{g-1}z)$.  
Then the following statements hold: 
\begin{enumerate}
\item $p_g(A)=g$. 
\item $\nr(I)=1$ and $\br(I)=g+1$. Indeed, 
\begin{enumerate}
\item $\overline{I}=I$ and $\overline{I^{n}}=I^n=QI^{n-1}$ 
for every $n=2,\ldots, g$. 
\item $\ell_A(\overline{I^{g+1}}/Q\overline{I^g})=1$
($\overline{I^{g+1}}= I^{g+1} +(xy^{g^2-1})$).
\item $\overline{I^{n+1}}=Q \overline{I^{n}}$ for every $n \ge g+1$.  
\end{enumerate}
\item $\bar{e}_0(I)=4g-2$, $\bar{e}_1(I)=3g-1$, $\bar{e}_2(I)=g$ and 
$\ell_A(A/I)=g$. 
\item $q(nI)=g-n$ for every $n=0,1,\ldots,g$; $q(gI)=q(\infty I)=0$. 
\end{enumerate}
\end{ex}

The first statement follows from that $a(A)=0$ and $g=g(\Proj (A))$.
For the convenience of the readers, 
we give a sketch of the proof  in the case of $g=2$ 
(see \cite[the proof of Example 3.10]{OWY5}). 
Let $A=K[y^2,yz,z^2,xy,xz]=R^{(2)}$ with $\deg x=3$ and 
$\deg y= \deg z=1$, and 
$I=(y^2,yz, z^4, xy,xz) \supset Q=(y^2-z^4,yz)$. 
Then one can easily see that $e_0(I)=\ell_A(A/Q)=4g-2=6$, 
$\ell_A(A/I)=p_g(A)=g=2$ and $I^2=QI$, $\overline{I}=I$. 
In particular, $\nr(I)=1$. 

%{\bf Claim 1.} 
\begin{clm}
$f_0 \in K[y,z]_{2n} \cap \overline{I^n} 
\Longrightarrow f_0 \in I^n$ for each $n \ge 1$. 
\end{clm}

\par 
The normality of $I_0=(y^2,yz,z^4)K[y,z] \subset K[y,z]$ implies the above 
claim.

\begin{clm}
%{\bf Claim 2.} 
$0 \ne f_1 \in K[y,z]_{2n-3}$, 
$xf_1 \in \overline{I^n} \Longrightarrow n \ge 3$. 
\end{clm}

\par 
By assumption and Claim 1, 
we have $(y^6+z^6)f_1^2=(xf_1)^2 \in \overline{I^{2n}} 
\cap K[y,z]_{2 \cdot 2n} \subset I^{2n}$. 
The degree (in $y$ and $z$) of any monomial in 
$I^{2n}=(y^2,  yz, z^4,  xy,  xz)$ is at least $4n=\deg (y^6+z^6) f_1^2$. 
Hence $(y^6+z^6)f_1^2 \in (y^2,  yz)^{2n}$ and 
the the highest power of $z$ appearing in $(y^6+z^6) f_1^2$ is 
at most $2n$.  Therefore $n \ge 3$.  

%{\bf Claim 3.} 
\begin{clm}
If $n \le 2$, then $\overline{I^n} \cap A_n \subset I^n \cap A_n$. 
\end{clm}

\par 
Any $f \in \overline{I^n} \cap A_n$ can be written as 
$f=f_0+xf_1$ for some $f_0 \in K[y,z]_{2n}$ and $f_1  \in K[y,z]_{2n-3})$. 
Let $\sigma \in \Aut_{K[y,z]^{(2)}}(A)$ 
such that $\sigma(x)=-x$. 
Then since $\sigma(I)=I$, we obtain $\sigma(f)=f_0 - xf_1 \in 
\overline{I^n}$. Hence 
\[
f_0 = \frac{f+\sigma(f)}{2} \in \overline{I^n} 
\quad \text{and} \quad  
xf_1 =\frac{f-\sigma(f)}{2} \in \overline{I^n}.
\] 
By Claims 1,2, we have $f_0 \in I^n$ and $f_1=0$. 
Therefore $f=f_0 \in I^n \cap A_n$, as required. 

%{\bf Claim 4.} 
\begin{clm}
$xy^3 \in  \overline{I^3} \setminus  Q\overline{I^2}$. 
\end{clm}

Since $(xy^3)^2=(y^6)^2+(y^3z^3)^2 \in (I^3)^2$, 
we get $xy^3 \in \overline{I^3}$.  
Assume $xy^3 \in Q\overline{I^2}=(a,b)\overline{I^2}$, 
where $a=y^2-z^4$ and $b=yz$. 
Then $axy+bxz^3=xy^3=au+bv$ for some $u,v \in \overline{I^2}$. 
Since $a,b$ forms a regular sequence, we can take an element 
$h \in A_1$ so that $u-xy=bh$ and $xz^3-v =ah$. So we may 
assume $u,v \in A_2$, and thus 
$u,v \in \overline{I^2} \cap A_2 \subset I^2$. 
However, this yields $xy^3=au+bv \in QI^2=I^3$, which is a 
contradiction.

%{\bf Claim 5.} 
\begin{clm}
$q(I)=1$, $q(2I)=q(\infty I)=0$, 
$\ell_A(\overline{I^3}/Q\overline{I^2})=1$ and 
$\overline{I^{n+1}}=Q\overline{I^n}$ for each $n \ge 3$.  
\end{clm}

By Proposition \ref{q(nI)}, we have 
$2=p_g(A)=q(0\cdot I)\ge q(I) \ge q(2\cdot I) \ge 0$.  
If $q(I)=q(2\cdot I)$, then $q(2 \cdot I)=q(3 \cdot I)$. 
This implies $\ell_A(\overline{I^3}/Q\overline{I^2})=0$ 
from Proposition \ref{eq:q(nI)}. This contradicts Claim 4. 
Hence $q(I)=1$ and $q(2\cdot I)=0$. 
The other assertions follow from Proposition \ref{eq:q(nI)}.  
In particular, $\br(I)=3$.

%{\bf Claim 6.} 
\begin{clm}
$\bar e_1 (I) =3g-2=5$, $\bar e_2(I)=g=2$. 
\end{clm}

By Proposition \ref{p:normalHP}, we have 
\begin{eqnarray*}
\bar e_1(I)&=& e_0(I)-\ell_A(A/I)+p_g(A)-q(I)=6-2+2-1=5, \\
\bar e_2(I) &=& p_g(A)-q(\infty I)=~2-0=2. 
\end{eqnarray*}
%
%%%%%%%%%%%%%%%%%%%%%%%%%%%%%%%%%%%%%%%%%%%%%%%%%%%%%%%%
%%%%   Section 3
\section{Elliptic and Strongly  Elliptic   ideals}

We define the  Rees algebra $\bar{\cR}(I)$ and the 
associated graded ring $\bar{G}(I)$ associated to the normal filtration 
as follows:

\begin{eqnarray*}
\bar{\cR}(I) 
&:= & \bigoplus_{n\ge 0}\overline{I^n} t^n \subset A[t]. \\
\bar{G}(I) &:= & \bigoplus_{n\ge 0}\overline{I^n}/ \overline{I^{n+1}}  \cong  
\bar{\cR}(I)/\bar{\cR}(I)(1). 
\end{eqnarray*}
\par \vspace{2mm}
$\bar{\cR}(I)$ (resp.  $\bar{G}(I)$) is called the 
\textit{normal Rees algebra}
(resp.  \textit{the normal associated graded ring}) of $I$.  
We recall that the $a$-invariant of a graded $d$-dimensional ring 
$R$ with maximal homogeneous graded ideal $\mathfrak{M}$ 
was introduced by \cite{GW} and 
defined as $a(R):= \max\{n | [H^d_{\MM}(R)]_n \neq 0\}$, 
where $[H^d_{\MM}(R) ]_n$ denotes  the homogeneous component of 
degree $n$ of the graded $R$-module $H^d_{\MM}(R)$. 
\par 
It is known that $A$ is a rational singularity if and only if 
$\overline r(A)=1$, see \cite[Proposition 1.1]{OWY5}. 
In \cite{OWY1,OWY2},  the  authors introduced the notion of $p_g$-ideals, 
characterizing rational singularities.   

%%%  Theorem 3.1
\begin{thm}[\textrm{cf. \cite{OWY1, OWY2, GN, Hun}}]  \label{r=1}
Let $(A,\m)$ be a two-dimensional excellent normal non-regular 
local domain containing an algebraically closed field $k= A/\m. $
Let $I=I_Z$ be an $\m$-primary integrally closed ideal of $A$. 
Put $\bar{G}=\bar{G}(I)$ and $\bar{\cR}=\bar{\cR}(I)$. 
Then the following conditions are equivalent$:$
\begin{enumerate}
\item $\br(I)= 1$. 
\item $q(I)=p_g(A)$. 
\item $I^2=QI$ and $\overline{I^n}=I^n$ for every $n \ge 1$. 
\item $\bar e_1(I) = e_0(I)-\ell_A(A/I)$. 
\item $\bar e_2(I)=0$. 
\item $\bar{G}$ is Cohen-Macaulay with $a(\bar{G})< 0$. 
\item $\bar{\cR}$ is Cohen-Macaulay. 
\end{enumerate}
When this is the case, $I$ is said to be a \textit{$p_g$-ideal}.  
\end{thm}

\begin{proof} 
Since $QI^{n-1} \subset I^n \subset \overline{I^n}$ 
for every $n \ge 2$, (1) $\Leftrightarrow (3)$ is trivial. 
(1) $\Leftrightarrow$ (5) (resp. (6) $\Leftrightarrow$ (7)) follows from 
\cite[Part II, Proposition 8.1]{GN} 
(resp. \cite[Part II, Corollary 1.2]{GN}).
Moreover. the equivalence of (4),(5) and (7) follows from 
\cite[Part II, Theorem 8.2]{GN}. 
(2) $\Leftrightarrow$ (4) follows from Proposition $\ref{eq:q(nI)}$. \qed
\end{proof}

\par 
It is known that $A$ is a rational singularity if and only if 
any integrally closed $\m$-primary ideal is a $p_g$-ideal
 (see \cite{OWY1, OWY2}).  
We define 
\[
\overline{r}(A):=\max\{\br(I) \,|\, \text{$I$ is an integrally closed 
$\m$-primary ideal}\}. 
\]
Then  $A$ is a rational singularity if and only if 
$\overline r(A)=1$, see \cite[Proposition 1.1]{OWY5}. 
 
\par \vspace{1mm}
Okuma proved in \cite[Theorem 3.3]{Ok}  
that if $A$ is an elliptic singularity, then $\overline r(A)=2$.
For the definition of elliptic singularity we refer to \cite[page 428]{W} 
or \cite[Definition 2.1]{Ok}. 
We investigate the integrally closed $\m$-primary ideals  such that 
$\br(I) =2$ with the aim to  characterize elliptic singularities.  
Next result extends and completes  a result by S. Itoh  
\cite[Proposition 10]{It2},  by using a different approach. 

%%%  Theorem  3.2
\begin{thm} \label{nr2}
Let $(A,\m)$ be a two-dimensional excellent normal local domain 
containing an algebraically closed field $k= A/\m$ and let 
$I \subset A$ be an $\m$-primary integrally closed ideal. 
Put $\bar{G}=\bar{G}(I)$ and $\bar{\cR}=\bar{\cR}(I)$. 
Then the following conditions are equivalent$:$
\begin{enumerate}
 \item $\br(I)=2$.   
 \item $p_g(A) > q(I)=q(\infty I)$. 
 \item $\bar{e}_1(I) = e_0(I) -\ell_A( A/I) + \bar{e}_2(I)$  
 and $\bar{e}_2(I) >0$.
 \item $\ell_A(A/\overline{I^{n+1}}) = \bar{P}_I(n)$ for all $n\ge 0$ and $\bar{e}_2(I) >0$.
 \item $\bar{G}$ is Cohen-Macaulay with $a(\bar{G})=0$.
\end{enumerate}
When this is the case, $I$ is said to be an \textbf{elliptic ideal}  and 
 $\ell_A ([H^2_{\MM}(\bar{G})_0) = \ell_A(\overline{I^2}/QI) = \bar{e}_2(I)$. 
 \end{thm} 
 
\begin{proof} 
$(1) \Longleftrightarrow (2):$ It follows from Proposition
$\ref{eq:q(nI)}$(2). 

\par \vspace{2mm} 
\noindent $(2) \Longleftrightarrow (3):$ 
By Proposition \ref{p:normalHP} we have
\begin{eqnarray*}
\bar{e}_1(I) &=& e_0(I) -\ell_A( A/I) + \bar{e}_2(I) 
- \bigg\{q(I)-q(\infty I)\bigg\}. \\
\bar{e}_2(I) &=& p_g(A) - q(\infty I) \ge 0. 
\end{eqnarray*}
The assertion follows from here. 

\par \vspace{2mm} 
$(2) \Longleftrightarrow (4):$   
Assume $I= I_Z = H^0(X,\mathcal{O}_X(-Z))$ 
for some resolution $X \to \Spec A$. 
By Kato's Riemann-Roch formula, for every integer $n \ge 0$, we have
\[
\ell_A(A/\overline{I^{n+1}})+h^1(\mathcal{O}_X(-(n+1)Z))
=-\dfrac{(n+1)^2Z^2+(n+1) K_XZ}{2} +p_g(A). 
\]
Hence 
\begin{eqnarray*}
\ell_A(A/\overline{I^{n+1}}) &=& \bar{e}_0(I){n+2 \choose 2} 
- \bar{e}_1(I) {n+1 \choose 1} + \bigg\{p_g(A)-q((n+1)I) \bigg\} \\
&=&  \bar{P}_I(n)- \bigg\{q((n+1)I)-q(\infty I) \bigg\}.
\end{eqnarray*}

\par 
Assume (4). 
By replacing  $0$ to  $n$ in the above equation, 
we get $q(I)=q(\infty I)$, hence (2). 
Conversely, if $q(I)=q(\infty I)$, then since $q((n+1)I)=q(\infty I)$ 
for all $n \ge 1$, the above equation implies (4). 

\par \vspace{2mm} 
\par \noindent $(1) \Longrightarrow (5):$   
Put $Q=(a,b)$. 
Since $\overline{I^{n+1}} \colon a=\overline{I^n}$, 
$a^{*}$, the image of $a$ in $\bar{G}$ is a non zero divisor of $\bar{G}$.
\par   
By assumption, we have  
$\overline{I^{n+1} } \cap Q = Q {\overline{I^{n}}} \cap Q
=Q {\overline{I^{n}}}$ for every $n \ge 2$. 
On the other hand, we have 
$\overline{I^2} \cap Q = QI$ 
by \cite[Theorem in page 371]{Hun} or \cite[Theorem]{It1}. 
Then it is well-known that
$a^{*}$, $b^{*}$ forms a regular sequence in $\bar{G}$, and thus 
$\bar{G}$ is Cohen-Macaulay (see also \cite{VV})
and  $2=\br(I)=a(\bar{G})+\dim A=a(\bar{G})+2$. 
Thus $a(\bar{G})=0$, as required. 

\par \vspace{2mm} 
\noindent $(5) \Longrightarrow (1):$ 
Since $\bar{G}$ is Cohen-Macaulay, 
we have $\br(I)=a(\bar{G})+\dim A=0+2=2$.   \qed
\end{proof}

\par 
We notice that if $A$ is \textit{not} a rational singularity, then 
elliptic ideals always exist.  

%%%%%%%   Proposition 3.3 
\begin{prop} \label{exist} Let $(A,\m)$ be a two-dimensional excellent 
normal local domain containing an algebraically closed field 
$k= A/\m$ 
 and let $I \subset A$ be an $\m$-primary integrally closed ideal 
which is not a $p_g$-ideal. Then there exists a positive integer $n$ such 
that $\overline{I^n} $ is an elliptic ideal. 
In particular, if $A$ is not a rational singularity, then for any 
$\m$-primary integrally closed ideal $I$ of $A$, 
then $\overline{I^{n}}$ is either a $p_g$-ideal or 
an elliptic ideal for every $n \ge p_g(A)$. 
\end{prop}

\begin{proof}  
Let $n$ be a positive integer such that  
$\ell_A(A/\overline{I^{n}})=\bar{P}_I(n-1)$. Since the integral closure of 
$(\overline{I^n})^{p}$ coincides with 
$\overline{I^{n p}}$ for $p$ large,  we have 
\[
\bar{e}_0(I^{n})= n^2 \bar{e}_0(I) ;  \ \ 
\bar{e}_1(I^{n}) = n \bar{e}_1(I)   + {n \choose 2}  \bar{e}_0(I); \ \ 
\bar{e}_2(I^n)= \bar{e}_2(I). 
\]
After substituting the  $\bar{e}_i(I^n)$'s  with the corresponding 
expressions in terms of the $\bar{e}_i(I)$'s  we conclude that
\begin{eqnarray*}
\bar{e}_2(I^n) -  \bar{e}_1(I^n) + \bar{e}_0(I^n)-\ell_A(A/\overline{I^n})
&=& \bar{e}_0(I) {n+1 \choose 2} - \bar{e}_1(I) n + \bar{e}_2(I) 
-\ell_A(A/\overline{I^n}) \\
&=& \bar{P}_I(n-1)  -\ell_A(A/\overline{I^n}) =0.
\end{eqnarray*}
Since $I$ is a not $p_g$-ideal, then $\bar{e}_2(I^n)= \bar{e}_2(I) >0. $   
Hence, by Theorem \ref{nr2}, then $\overline{I^n}$ is an elliptic ideal. 
\qed
\end{proof} 

\par 
We denote by $\mathfrak{M} = \m + \bar{\cR}_{+} $ the homogeneous maximal 
ideal of  $\bar{\cR}$. 
As usual we say that $\bar{\cR}$  is $($FLC$)$ if 
$\ell_A(H^i_{\mathfrak{M}} (\bar{\cR}))< \infty $ for every $i \le \dim A=2$.  

%%%%%   Proposition 3.4
\begin{prop} \label{H} 
Assume  $I$ is an elliptic ideal,  then 
$\bar{\cR}$  is $($FLC$)$ but not Cohen-Macaulay with 
\[
H^2_{\MM}(\bar{\cR}) = [H^2_{\MM}(\bar{\cR})]_0
\cong [H_{\MM}^2(\bar{G})]_0. 
\]
\end{prop}

\begin{proof} 
%\par \vspace{2mm} 
Note that $\bar{\cR}_{\MM}$ is a universally catenary domain which is a homomorphic image of a Cohen-Macaulay local ring. 
Hence it is an (FLC) because $\bar{\cR}$ satisfies Serre condition $(S_2)$. Thus $H_{\MM}^0(\bar{\cR})=  H_{\MM}^1(\bar{\cR})=0$ and
$H_{\MM}^2(\bar{\cR})$ has finite length. 
\par 
Put $\mathcal{N}=\bar{\cR}_{+}$.   
Then we obtain two exact sequences of graded $\bar{\cR}$-modules. 
\[
0 \to \mathcal{N} \to \bar{\cR} \to {}_h A \to 0, \; 
\]
\[
 0 \to \mathcal{N}(1) \to \bar{\cR} \to \bar{G} \to 0, \; 
\]
where ${}_h A$ can be regarded as $\bar{\cR}/\mathcal{N}$ 
which is concentrated in degree $0$. 
One can easily see that $H_{\MM}^0(\mathcal{N})=H_{\MM}^1(\mathcal{N})=0$, and we get 
\begin{equation} \label{firstEx}
0 \to H_{\MM}^2(\mathcal{N}) \to H_{\MM}^2(\bar{\cR}) \to 
{}_h H_{\m}^2(A) \to H_{\MM}^3(\mathcal{N}) \to H_{\MM}^3(\bar{\cR}) \to 0, \;  
\end{equation}
\begin{equation} \label{secondEx}
0 \to 
H_{\MM}^2(\mathcal{N})(1) \to H_{\MM}^2(\bar{\cR}) \to 
H_{\MM}^2(\overline{G}) \to H_{\MM}^3(\mathcal{N})(1) \to H_{\MM}^3(\bar{\cR}) \to 0. \;   
\end{equation}
For any integer $n \le -1$, the first exact sequence (\ref{firstEx}) 
yields 
\[
0\to [H_{\MM}^2(\mathcal{N})]_n \to [H_{\MM}^2(\bar{\cR})]_n \to 0. \;  
\] 
Also, the second exact sequence (\ref{secondEx}) yields 
\[
0=[H_{\MM}^1(\bar{G})]_n \to  [H_{\MM}^2(\mathcal{N})]_{n+1} \to [H_{\MM}^2(\bar{\cR})]_n.  \;  
\]  
Then 
$[H_{\MM}^2(\bar{\cR})]_{-1} \subset [H_{\MM}^2(\bar{\cR})]_{-2} 
 \subset \cdots \subset [H_{\MM}^2(\bar{\cR})]_{n}=0$ for $n \ll 0$ and thus $[H_{\MM}^2(\bar{\cR})]_{n}=0$ for all $n \le -1$. 
\par \vspace{2mm} 
For any integer $n \ge 1$, the first exact sequence (\ref{firstEx}) 
yields 
\[
0 \to [H_{\MM}^2(\mathcal{N})]_n \to [H_{\MM}^2(\bar{\cR})]_n \to 0. \; 
\] 
Moreover, as $a(\bar{G})=0$, we have 
\[
[H_{\MM}^2(\mathcal{N})]_{n+1} \to [H_{\MM}^2(\bar{\cR})]_n \to 
[H_{\MM}^2(\bar{G})]_n =0 \; \text{(ex)}.  
\]
Hence we get $[H_{\mathcal{M}}^2(\bar{\cR})]_n=0$ for all $n \ge 1$. 
\par 
Since $a(\bar{\cR})=-1$, we have 
$[H_{\MM}^3(\mathcal{N})]_1 \cong [H_{\MM}^3(\bar{\cR})]_1=0$. 
Hence we get 
\[
H_{\MM}^2(\bar{\cR})=
[H_{\MM}^2(\bar{\cR})]_0 \cong [H_{\MM}^2(\bar{G})]_0, 
\]
as required. \qed
\end{proof}

%%%%    Corollary 3.5
\begin{cor} \label{cohomology} 
Let $(A,\m)$ be a two-dimensional excellent normal local domain and let 
$I \subset A$ be an $\m$-primary integrally closed ideal. 
\par \noindent
Then $I$ is an elliptic ideal  if and only if 
$0 \ne H^2_{\MM}(\bar{\cR}) = [H^2_{\MM}(\bar{\cR})]_0 \hookrightarrow H_{\m}^2(A)$, where the last map is induced from the natural surjection 
$\bar{\cR} \to {}_h A = \bar{\cR}/\bar{\cR}_{+}$. 
\end{cor} 

\begin{proof} Assume $I$ is an elliptic ideal, then from the proof of Proposition \ref{H} and Theorem \ref{nr2}  
we conclude our assertions. 
Conversely, by our assumption, we can conclude that 
$\bar{G}(I)$ is Cohen-Macaulay with $a(\bar{G}(I))=0$ 
by a similar argument as in the proof of Proposition \ref{H}. 
Hence  $I$ is an elliptic ideal by Theorem \ref{nr2}. \qed
\end{proof}  

\par 
For a cycle $C>0$ on $X$, we denote by $\chi(C)$ the Euler 
characteristic of $\cO_C$. 

%%%%%   Definition 3.6
\begin{defn} \label{fund}
Let $Z_f$ denote the \textit{fundamental cycle}, namely, 
the non-zero  minimal anti-nef cycle on $X$. 
The ring $A$ is called \textit{elliptic} if $\chi(Z_f)=0$. 
\end{defn}

\par 
The following result  follows from Theorem \ref{nr2} and  \cite{Ok}, 
Theorem $3.3$. 

%%%%%%   Corollary 3.7
\begin{cor} \label{elliptic}
If $A$ is an elliptic singularity, then for every integrally closed ideal $I \subset A$ the following facts hold:
\begin{enumerate}
\item  $\bar{G}(I)$ is Cohen-Macaulay with  $a(\bar{G}(I))\le 0$.
\item $I$ is elliptic or a $p_g$-ideal.
\end{enumerate}
Since there always exists an ideal $I$ with $q(I)=0$, we have 
$\bar r(A)=2$.  
\end{cor}

\par
The result above gives some evidence about a positive answer to the following question:
 
%%%%%%   Question 3.8
\begin{quest} \label{ell} Assume $\bar r(A) =2$, 
is it true that $A$ is an elliptic singularity?
\end{quest}

\par  
We can give a positive answer to Question \ref{ell} if 
$ \bar{e}_2(I) \le 1$ for all integrally closed $\m$-primary ideals.  
 In the following result we describe the integrally closed 
 $\m$-primary ideals satisfying this minimal condition.  

%%%%%  Theorem 3.9
\begin{thm}\label{e_2=1}
Let $(A,\m)$ be a two-dimensional excellent normal local domain 
over an  algebraically closed field.  
Let $I \subset A$ be an $\m$-primary integrally closed ideal, 
and let $Q$ be a minimal reduction of $I$.   
Put $\bar{G}=\bar{G}(I)$ and $\bar{\cR}=\bar{\cR}(I)$. 
Then the following conditions are equivalent$:$
\begin{enumerate}
 \item $\bar{r}(I)=2$ and $\ell_A(\overline{I^2}/Q I) =1$.  
 \item $q(I)=q(\infty I)=p_g(A)-1$. 
 \item $\bar{e}_2(I) =1$.
 \item $\bar{e}_1(I) = e_0(I) -\ell_A( A/I) + 1$ and $\nr(I)=\br(I)$. 
 \item $\bar{G}$ is Cohen-Macaulay with $a(\bar{G})=0$ and 
 $\ell_A([H^2_{\MM}(\bar{G})]_0)=1$. 
\end{enumerate}
When this is the case,  $I$ is said to be a \textbf{strongly elliptic ideal} 
and $\bar{\cR}$  is a Buchsbaum  ring  with 
$\ell_A(H^2_{\MM}(\bar{\cR})) = 1$.   
\end{thm}

\begin{proof}
$(1) \Longrightarrow (2):$ By Theorem \ref{nr2}, we have 
$p_g(A) > q(I)=q(\infty I)$.  
In particular, $q(2 I)=q(I)$. 
By Proposition \ref{eq:q(nI)}(1), $p_g(A)-q(I)=\ell_A(\overline{I^2}/QI)=1$.  
Conversely $(2) \Longrightarrow (1) $ again by Proposition \ref{eq:q(nI)}. 

\par \vspace{2mm} 
$(2) \Longrightarrow (3):$ 
By Proposition \ref{p:normalHP}(4), we have 
\[
\bar{e}_2(I)=p_g(A)-q(I)=1. 
\]

\par \vspace{2mm} 
$(3) \Longrightarrow (2):$ 
Since $p_g(A)-q(\infty I)=\bar{e}_2(I)=1$ by assumption, we have 
$p_g(A)-1 = q(\infty I) \le q(I) \le p_g(A)$. 
If $q(I)=p_g(A)$, then $I$ is a $p_g$-ideal and thus $\bar{e}_2(I)=0$. 
This is a contradiction. Hence $q(\infty I) = q(I)=p_g(A)-1$, as required. 

\par \vspace{2mm} 
$(1),(3) \Longrightarrow (4):$ 
It follows  from Theorem \ref{nr2} $(1)\Longrightarrow (3)$ and the fact that $1< \nr(I) \le \br(I)=2. $

\par \vspace{2mm} 
$(4) \Longrightarrow (1):$
By Proposition \ref{eq:q(nI)}(1), we have 
\begin{eqnarray*}
\ell_A(\overline{I^2}/QI)&=& (p_g(A)-q(I))-(q(I)-q(2I)), \\
\ell_A(\overline{I^3}/Q\overline{I^2})&=& (q(I)-q(2I))-(q(2I)-q(3I)), \\
\vdots \qquad ~  &=& ~ \qquad  \vdots 
\end{eqnarray*}
By a similar argument as in \cite{Hun} 
and Proposition \ref{p:normalHP}, we get
\begin{eqnarray*}
\bar{e}_2(I) &=& \sum_{n=1}^{\infty} n \cdot \ell_A(\overline{I^{n+1}}/Q\overline{I^n}), \\[2mm] 
\bar{e}_1(I)-\bar{e}_0(I)+\ell_A(A/I)
&=& \sum_{n=1}^{\infty} \ell_A(\overline{I^{n+1}}/Q\overline{I^n}). 
\end{eqnarray*}
Thus our assumption implies $\ell_A(\overline{I^{n+1}}/Q\overline{I^n})=1$ 
for some unique integer  $n \ge 1$. 
On the other hand, since $\nr(I)=\br(I)$, we must have $n=1$. 

\par \vspace{2mm} 
$(1) \Longrightarrow (5):$ 
Suppose (1). Then Theorem \ref{nr2}$(1)\Longrightarrow (5)$ 
implies that $\bar{G}$ is Cohen-Macaulay with $a(\bar{G})=0$.
\par
We remark that $\sqrt \MM = \sqrt {\bar G_+} $ in $\bar G, $ hence by \cite[Proposition 3.1]{MSV}, we have 
$[H_{\MM}^2(\bar G)]_0 \cong \overline{I^2}/QI \cong A/\m$
has length $1$. 
In particular by Proposition \ref{H},  $H_{\MM}^2(\bar{\cR})$ 
becomes a $A/\m$-vector space 
and thus $\bar{\cR}$ is Buchsbaum. 

\par \vspace{2mm} 
$(5) \Longrightarrow (1):$ 
By Theorem \ref{nr2}$(5) \Longrightarrow (1)$, we have $\br(I)=2$. 
Also, $\ell_A(\overline{I^2}/QI)=\ell_A([H_{\MM}^2(\bar{G})]_0)=1$. \qed
\end{proof}

\par
It is clear that if $I$ is a strongly elliptic ideal, then $I$ is an elliptic ideal. 
In some cases they are equivalent. 
Notice that the converse is \textit{not} true in general. 
For instance, let $A=k[[x^2,y^2,z^2,xy,xz,yz]]/(x^4+y^4+z^4)$. 
Then $A$ is a $2$-dimensional normal local domain with the maximal ideal 
$\m=(x^2,y^2,z^2,xy,xz,yz)$. Then $\m$ is a normal ideal and 
$Q=(x^2,y^2)$ is a minimal reduction of $\m$ with $\m^3=Q\m^2$. 
Moreover, $\br(\m)=r(\m)=2$ and $\ell_A(\m^2/Q\m)=3$ imply 
that $\m$ is an elliptic ideal but not a strongly elliptic ideal. 

Notice that (1) is equivalent to (3) follows also from \cite{It3}.

%%%%%%%   Proposition 3.10 
\begin{prop} \label{str} Let $(A,\m)$ be a two-dimensional Gorenstein 
excellent normal local domain. Then $\m$ is an elliptic ideal  if and only if 
$\m$ is a strongly elliptic ideal. 
\end{prop} 

\begin{proof} 
Assume $\m$ is an elliptic ideal  and $Q$ be
its minimal reduction. Since $\bar{r}(\m)=2$, 
$\m \overline{\m^2} \subset Q$ and we have 
$\overline{\m^2}/Q\m  \cong (\overline{\m^2}+Q)/Q 
\hookrightarrow A/Q$, whose image is contained in $(Q:\m)/Q$.
Since the latter has length $1$, $\ell_A (\overline{\m^2}/Q\m) =1$
and $\m$ is strongly elliptic.
\end{proof}

%%%%%%   Example 3.11
\begin{exam}  \label{Brieskorn}
Let $A=\mathbb{C}[[x,y,z]]/(x^a+y^b+z^c)$ be a Brieskorn hypersurface, 
where $2 \le a \le b \le c$.  
Then:  
\begin{enumerate}
\item $\m$ is a $p_g$-ideal if and only if $(a,b)=(2,2),(2,3)$. 
\item $\m$ is an elliptic ideal (equivalently strongly elliptic) if and only if 
\[
(a,b)=(2,4),(2,5),(3,3),(3,4). 
\]
\end{enumerate}
In particular, if $p \ge 1$ and $(a,b,c)=(2,4,4p+1)$, then $p_g(A)=p$ and 
$\m$ is a (strongly) elliptic ideal. 
It follows from \cite[Theorem 3.1, Proposition 3.8]{OWY4}. 
\end{exam} 

%%%%%%   Example 3.12
\begin{exam} 
Proposition \ref{str} does not hold if $I \neq  \m$. 
Let $A$ be any two-dimensional excellent normal local domain
with $p_g(A) > 1$.  Then there exist 
always integrally closed ideals $I$ with $q(I)=0$.  
Since $q(I)=q(2I)=0$, $\bar{r}(I) =2$ and $\bar{e}_2(I) = p_g(A)$. 
Thus \ref{str} does not hold for such $I$. 
\end{exam} 

\par 
We recall that an excellent normal local domain for which every integrally closed $\m$-primary ideal is a $p_g$-ideal,  is a rational singularity ($p_g(A)=0$). This result suggests to study the next step.

%%%%%  Definition 3.13
\begin{defn}[\textrm{e.g. \cite{Y}}] An excellent normal 
local domain $A$ is a   {\it{strongly elliptic singularity}} if  $p_g(A)=1$.  
\end{defn} 

\par
Note that any strong elliptic singularity is an elliptic singularity. 
The following result characterizes algebraically  the strongly 
 elliptic singularities. 

%%%%%  Theorem 3.14
\begin{thm} \label{pg=1} Let $(A,\m)$ be a two-dimensional excellent 
normal local domain containing an algebraically closed field $k= A/\m$ 
 and assume that $p_g(A)>0$. 
The following facts are equivalent:
\begin{enumerate}
\item $A$ is a strongly elliptic singularity. 
\item Every integrally closed ideal of $A$ is 
either a $p_g$-ideal or  a  strongly  elliptic ideal.
\end{enumerate}
\end{thm}

\begin{proof} 
It depends by the fact that always there exists an integrally closed ideal $I$ of $A$  such that $q(I)=0$.
Thus  $p_g(A)= \bar{e}_2(I)$.  \qed
\end{proof} 

\par
If $A$ is a rational singularity, then every integrally closed 
$\m$-primary ideal is normal. 
This is not  true if $A$ is an elliptic singularity, 
even if we assume $A$ is a strongly  elliptic singularity. 
 
%%%  Example 3.15
\begin{ex} \label{exnormal}   
\begin{enumerate}
\item Let $A = k[ X,Y,Z]/ (X^3 + Y^3 + Z^3), $ then $A$ is Gorenstein, 
$p_g(A) = 1$ and the maximal ideal $\m$ is normal.  
If we consider  $I = (x,y, z^2), $ then  $I^2$  is not normal.
 \item Cutkosky showed that if 
$A= \mathbb{Q}[[ X,Y,Z]]/ (X^3 +3 Y^3 +9 Z^3)$,  
($\mathbb{Q}$  rational numbers),  
then for every integrally closed ideal $I \subset A$, 
$I^2 $ is also integrally closed and hence normal.  
This is because the elliptic curve does not have 
any $\mathbb{Q}$-rational point.
\item Let $A=k[x,y,z]/(x^2+y^4+z^4)$, $I=\m=(x,y,z)$, $Q=(y,z)$.  
Then $p_g(A)=1$ and $ \overline{\m^n}= x(y,z)^{n-2}+\m^n$  
for every $n \ge 2$. 
\end{enumerate}
\end{ex}

%%%%%   Proposition 3.16
\begin{prop} \label{normal} 
Let $(A,\m)$ be a two-dimensional excellent normal local domain 
containing an algebraically closed field $k= A/\m. $
Assume that $I$ is a strongly elliptic ideal. 
Then the following conditions are equivalent$:$
\begin{enumerate}
\item $\overline{I^2}=I^2$. 
\item $\overline{I^n}=I^n$ for some $n \ge 2$. 
\item $\overline{I^n}=I^n$ for every $n \ge 2$. 
\end{enumerate}
\end{prop}

\begin{proof} By Theorem  \ref{e_2=1} (1), we have 
$\ell_A(\overline{I^2}/QI)=1$ and 
$\overline{I^n} = Q \overline{I^{n-1}}$ for $n\ge 3$. 
Hence if $I^2 =  \overline{I^2}$, then 
$I^n =  \overline{I^n}$ for all $n\ge 2$. \par
\par 
Conversely, assume that $I^2 \ne \overline{I^2}$. 
Since $\ell_A(\overline{I^2}/QI)=1$, we should have 
$I^2 = QI$. This implies that $G(I) := 
\oplus_{n\ge 0} I^n/ I^{n+1}$ is Cohen-Macaulay with 
$a(G(I)) = -1$   ([\cite{VV}, Proposition 2.6] and \cite{GS}) and hence 
\[
\ell_A( A/ I^{n+1}) =  
 e_0(I) \genfrac{(}{)}{0pt}{0}{n+2}{2}
- e_1(I)\genfrac{(}{)}{0pt}{0}{n+1}{1} 
\] 
with $e_0(I) = \bar e_0(I)$ and $e_1(I) = e_0(I) - \ell_A(A/I)$.
\par 
On the other hand, by Theorem \ref{nr2} and Corollary $\ref{e_2=1}$, we have   
\begin{eqnarray*}
\ell_A(A/\overline{I^{n+1}})&=& 
\bar{P}_I(n)=\overline{e}_0(I){n+2 \choose 2} - \overline{e}_1(I){n+1 \choose 1} + \overline{e}_2(I) \\
&=& e_0(I){n+2 \choose 2} - (e_0(I)-\ell_A(A/I)+1){n+1 \choose 1} +1 \\
&=& \ell_A(A/I^{n+1})-n.
\end{eqnarray*}
This implies that $I^n \ne \overline{I^n}$ for all $n\ge 2$. \qed
\end{proof}

\par  
We can characterize the normal ideals in a strongly elliptic singularity. 
Before showing the results, let us recall some definitions and basic facts on cycles and a vanishing theorem for elliptic singularities.
In the following, $A$ is an elliptic singularity 
and $X$ is a resolution of  $\Spec(A)$. 
\par
For a cycle $C>0$ on $X$, we denote by $\chi(C)$ 
the Euler characteristic $\chi(\cO_C)=h^0(\cO_C)-h^1(\cO_C)$. 
Then $p_a(C):=1-\chi(C)$ is called the {\em arithmetic genus} of $C$.
By the Riemann-Roch theorem, we have $\chi(C)=-(K_X+C)C/2$, where $K_X$ is the canonical divisor on $X$. 
From this, if $C_1, C_2>0$ are cycles, we have $\chi(C_1+C_2)=\chi(C_1)+\chi(C_2)-C_1C_2$. 
From the exact sequence 
\[
0\to \cO_{C_2}(-C_1) \to \cO_{C_1+C_2} \to \cO_{C_1} \to 0
\]
we have $\chi(\cO_{C_2}(-C_1))=-C_1C_2+\chi(C_2)$. 
\par
If $A$ is elliptic, then there exists a unique cycle $E_{min}$, 
called the {\em minimally elliptic cycle}, such that $\chi(E_{min})=0$ and
$\chi(C)>0$ for all cycles $0<C<E_{min}$ (see \cite{La}).
Moreover we have the following 
(see \cite[Proposition 3.1, Proposition 3.2, Corollary 4.2]{La}, \cite[p.428]{W}, \cite[(6.4), (6.5)]{T}).

%%%%%%%  Proposition 3.17
\begin{prop} \label{p:ellchi}
Assume that $A$ is elliptic. 
Then $\chi(C)\ge 0$ for any cycle $C>0$ on $X$ and 
$C\ge E_{min}$ if $\chi(C)=0$.
\end{prop}

\par 
Let us recall that the fundamental cycle $Z_f$ can be computed 
via a sequence of cycles:
\[
C_0:=0, \quad C_1=E_{j_1}, \quad C_i=C_{i-1}+E_{j_i}, \quad C_m=Z_f,
\]
where $E_{j_1}$ is an arbitrary component of $E$ and $C_{i-1}E_{j_i}>0$ 
for $2\le i \le m$.
Such a sequence $\{C_i\}$ is called a  
 {\em computation sequence} for $Z_f$. 
 It is known that $h^0(\cO_{C_i})=1$ for $1 \le i\le m$ 
 (see \cite[p.1260]{La}).

\par
The following vanishing theorems are essential in our argument.

%%%%%   Theorem 3.18
\begin{thm}[R{\"o}hr {\cite[1.7]{Rh}}] \label{t:rohr}
Let $L$ be a divisor on $X$ such that $LC>-2\chi(C)$ 
for every cycle $C>0$ which occurs in a computation sequence 
for $Z_f$.
Then $H^1(\cO_X(L))=0$.
If $A$ is rational, then the converse holds, too.
\end{thm}

\par 
From Theorem \ref{t:rohr} and Proposition \ref{p:ellchi}, 
we have the following.

%%%%%%%  Corollary 3.19
\begin{cor}\label{c:ellv}
Assume that $A$ is an elliptic singularity. 
Let $L$ be a nef divisor on $X$ such that $L  E_{min}>0$.
Then $H^1(\cO_X(L))=0$.
\end{cor}

%%%%%%%   Proposition 3.20
\begin{prop}\label{p:c1}
Assume that $A$ is an elliptic singularity 
and $D$ the minimally elliptic cycle on $X$. 
Let $F$ be a nef divisor on $X$.
If $FD>0$, then $H^1(\cO_X(F-D))=0$, 
and from the exact sequence 
$0 \to \cO_X(F-D) \to \cO_X(F) \to \cO_D(F) \to 0$,  
the restriction map
$H^0(\cO_X(F))\to H^0(\cO_D(F))$ is surjective.
\end{prop}

\begin{proof}  
If $F-D$ is nef, since $(F-D)D>0$, 
we have $H^1(\cO_X(F-D))=0$ by Corollary \ref{c:ellv}.
Assume that $F-D$ is not nef.
As in \cite[1.4]{G}, we have a sequence $\{D_i\}$ of cycles such that
\[
D_0=D, \ \ D_i = D_{i-1} + E_{j_i}, \; (F-D_{i-1})  E_{j_i}<0\; 
(1\le i \le s), \ \ F-D_{s} \text{ is nef.}
\]
Since $F-Z_f$ is nef, $(F-D_{i-1})  E_{j_i}<0$ implies $D_{i-1}E_{j_i}>0$
 and $D\le Z_f$, we see that $D_s\le Z_f$ and $D_s$ occurs 
 in a computation sequence for $Z_f$.
Then the equalities 
$\chi(D)=\chi(D_s)=0$ and $\chi(D_i) = \chi(D_{i-1}) 
+\chi( E_{j_i})-D_{i-1}  E_{j_i}$ 
imply that $F  E_{j_i}=0$, $D_{i-1}  E_{j_i}=1$, 
and $h^j(\cO_{E_{j_i}}(F-D_{i-1}))=0$ for $j=0,1$ and $1\le i \le s$.
Since 
\[
0\le \chi(D_s+D)=\chi(D_s)+\chi(D)-DD_s=-DD_s,
\]
we have $(F-D_s)D>0$.
Therefore, from the exact sequence
\[
0 \to \cO_X(F-D_i) \to \cO_X(F-D_{i-1})
\to \cO_{E_{j_i}}(F-D_{i-1}) \to 0,
\]
we obtain $H^1(\cO_X(F-D))=H^1(\cO_X(F-D_s))=0$. \qed
\end{proof}

%%%%%%%%  Theorem 3.21
\begin{thm}[{\cite[2.7]{FT}}] \label{t:FT}
Let $C$ be a Cohen-Macaulay projective scheme 
of pure dimension $1$, and let $\cF$ be a rank $1$ 
torsion-free sheaf on $C$.
Assume that $\deg \cF|_{W}:= \chi(\cF|_{W})-\chi(W) > -2\chi(W)$ 
for every subcurve $W\subset C$.  
Then $H^1(\cF)=0$. 
\end{thm}

\par
To show the normality of an ideal $I$, the following is essential. 

%%%%%%%%  Proposition 3.22
\begin{prop} \label{p:c2}
Let $\cL_1$ and $\cL_2$ be nef invertible sheaves 
on the minimally elliptic cycle $D$ such that 
$d_i:=\deg \cL_i \ge 3$ for $i=1,2$.
Then the multiplication map
\[
\gamma\: H^0(\cL_1)\otimes H^0(\cL_2) \to 
H^0(\cL_1\otimes \cL_2)
\]
is surjective.
\end{prop}

\begin{proof}
First, note that $\chi(W)>0$ for any cycle $0<W\lneqq D$ 
by the definition of the minimally elliptic cycle.
For any subscheme $\Lambda \subset D$, 
we denote by $I_{\Lambda}\subset \cO_D$ 
the ideal sheaf of $\Lambda$.
For any cycle $W\le D$ and any $p\in \supp (W)$, we have 
$\deg(I_p\cL_i)|_W=\deg \cL_i|_W-1$.  Therefore, it follows from 
Theorem \ref{t:FT} that $H^1(I_p\cL_i)=0$ for any point 
$p \in \supp(D)$.
Hence $\cL_i$ is generated by global sections. 
Let $s\in H^0(\cL_1)$ be a general section and 
consider the exact sequence
\begin{equation}
\label{eq:s}
0\to \cO_D \xrightarrow{ \times s } \cL_1 \to \cL_1|_B \to 0,
\end{equation}
where $B$ is the zero-dimensional subscheme of $D$ 
of degree $d_1=\deg \cL_1$ defined by $s$.
Note that since $\cL_1$ is generated by global sections, 
each point of $\supp (B)$ is a nonsingular point 
of $\supp (D)$ and there exists $s_1\in H^0(\cL_1)$ 
such that $\cL_1|_B\cong s_1\cO_B\cong \cO_B$.
Let $p\in  \supp(B)$ be any point.
The following fact make our proof easier.

\begin{clm}\label{cl:1}
Let $\n \subset \cO:=\cO_{X,p}$ be the maximal ideal.
Then we can take generators $x,y$ of $\n$ so that 
$\cO_{D,p}=\cO/(x^{n_p})$ with $n_p\ge 1$ and 
$\cO_{B,p}\cong \cO/(x^{n_p},y)$.
Hence, at $p$, the subschemes of $B$ correspond to monomials 
$x^{\ell}$ with $\ell\le n_p$.
\end{clm}

\begin{proof}[Proof of Claim \ref{cl:1}]
Since $E$ is nonsingular at $p$, we have the generators $x,y\in \n$ 
such that $\cO_{D,p}=\cO/(x^{n_p})$. 
Assume that $\cO_{B,p}=\cO/(x^{n_p}, f)$, 
where $f\in \n$.
If $\ell_{\cO}(\cO/(x,f))=1$, we can put $y=f$.
Assume that  $\ell_{\cO}(\cO/(x,f))\ge 2$.
Let $0<W\le D$ be any cycle and let $\cO_{W,p}=\cO/(x^{m_p})$. 
Assume that $p\in \supp (W)$.
Then the cokernel of $(I_p^2\cL_1)|_W\to \cL_1|_W$ 
is isomorphic to $\cO/\n^2+(x^{m_p})$.
If $m_p=1$, then $\deg(I_p^2\cL_1)|_W=\deg\cL_1|_W-2$.
If $m_p\ge 2$, then  $\deg(I_p^2\cL_1)|_W
=\deg\cL_1|_W-3\ge \ell_{\cO}(\cO/(x^{m_p},f))-3\ge 1$. 
Thus we have $H^1(I_p^2\cL_1)=0$ by Theorem \ref{t:FT} and the  map 
$H^0(I_p\cL_1) \to I_p\cL_1/I_p^2\cL_1$ is surjective; however, 
this shows that we can take $f=y$. \qed
\end{proof}

\par
Tensoring $\cL_2$ with the sequence Eq. $(\ref{eq:s})$, 
we obtain the exact sequence 
\[
0\to H^0( \cL_2) \xrightarrow{ \times s } 
H^0(\cL_1 \otimes \cL_2) \to H^0(\cO_B) \to 0,
\]
since $H^1(\cL_2)=0$.
As seen above, we have general sections $s_1\in H^0(\cL_1)$ and  
$s_2\in H^0(\cL_2)$ such that $s_1s_2\mapsto 1\in H^0(\cO_B)$.
Thus the sections of $H^0(\cL_1 \otimes \cL_2)$ 
which map to $1\in H^0(\cO_B)$ are in the image of $\gamma$.
It is now sufficient to show that for any subscheme $B'\subset B$ 
of $\deg B'<d_1$, the image of $\gamma$ contains a section 
$t \in H^0(\cI_{B'}\cL_1\otimes \cL_2)$ such that 
$\cL_1\otimes \cL_2/t\cO_D \cong \cO_{B'}\oplus\cO_{\overline{B}}$, 
where $\supp (\overline{B}) \cap \supp (B)=\emptyset$.
To prove this, we write $\cI_{B'}=\cI_{B_1}\cI_{B_2}$ 
($B_1, B_2\subset B'$) so that $\deg \cI_{B_i}\cL_i=d_i-\deg B_i\ge 2$ 
for $i=1,2$ (note that $\deg B_1 + \deg B_2 <d_1$).
Let $0<W\le D$ be any cycle and  $p\in \supp(B)$.  
We use the notation of the proof of Claim \ref{cl:1}. 
Suppose that $\cO_{B_1,p}=\cO/(x^{\ell_p},y)$.
Then $\deg \cL_1|_W=\sum_W m_p$, 
where $\sum_W$ means the sum over 
$p\in \supp(B) \cap \supp(W)$, and the cokernel of 
$(I_{B_1}\cL_1)|_W\to \cL_1|_W$ is isomorphic to 
$\cO/(x^{m_p}, x^{\ell_p},y)$.
Therefore, 
\[
\deg(I_{B_1}\cL_1)|_W=\sum_W (m_p-\min(m_p, \ell_p)).
\]
Since $\deg \cI_{B_1}\cL_1=d_1-\deg B_1\ge 2$, by Theorem \ref{t:FT}, 
we have $H^1(I_q\cI_{B_1}\cL_1)=0$ for any point $q\in \supp(B)$. 
Hence $H^0(\cI_{B_1}\cL_1)$ has no base points.
Clearly, the same results for $\cI_{B_2}\cL_2$ hold.
Therefore, for each $i=1,2$,
we have a section $t_i\in H^0(\cI_{B_i}\cL_i)$ such that 
 $\cL_i/t_i\cO_D \cong \cO_{B_i}\oplus\cO_{\overline{B_i}}$, 
where  $\supp (\overline{B_i}) \cap \supp (B)=\emptyset$.
Then $t:=t_1t_2$ satisfies the required property. \qed
\end{proof}

%%%%%%%% Theorem 3.23
\begin{thm} \label{I2}  
Let $(A,\m)$ be a two-dimensional excellent normal local domain 
containing an algebraically closed field $k= A/\m$. 
Assume that $A$ is a strongly elliptic singularity. 
If $I = I_Z $ is an elliptic ideal   
(equivalently $I$ is not a $p_g$-ideal) and $D$ is 
the minimally elliptic cycle on $X$, 
then $I^2$  is integrally closed (equivalently $I$ is normal) 
if and only if $- Z D \ge 3$  and 
if $- ZD \le 2$,   then $I^2 = QI$.    
\end{thm} 
 
\begin{proof} 
Assume that $-ZD\le 2$.
Since $H^1(\cO_X(-Z))=0$, by the Riemann-Roch theorem, we have 
$h^0(\cO_D(-nZ))=-nZD$ for $n\ge 1$. 
Hence 
\[
H^0(\cO_D(-Z))\otimes H^0(\cO_D(-Z)) \to H^0(\cO_D(-2Z))
\] 
cannot be surjective. 
By Proposition \ref{p:c1}, the map 
\[
H^0(\cO_X(-Z))\otimes H^0(\cO_X(-Z)) \to H^0(\cO_X(-2Z))
\]
cannot be surjective, too.
Therefore, $I^2\ne I_{2Z}$, and hence $I^2 = QI$. 

\par \vspace{2mm}
Assume that  $-ZD\ge 3$.
By Propositions \ref{p:c1} and \ref{p:c2}, we have 
the following commutative diagram:
\[
\begin{CD}
H^0(\cO_X(-Z))\otimes H^0(\cO_X(-Z)) @>{\alpha}>> 
H^0(\cO_X(-2Z)) \\
@VVV @VVV \\
H^0(\cO_D(-Z))\otimes H^0(\cO_D(-Z)) @>>> H^0(\cO_D(-2Z)),
\end{CD}
\]
where at least the maps other than $\alpha$ are surjective.
By Proposition \ref{p:c1} and its proof, we have 
$I_{2Z}=I^2+H^0(\cO_X(-2Z-D))$, 
$H^0(\cO_X(-2Z-D))=H^0(\cO_X(-2Z-D_s))$, 
and $-(Z+D_s)D\ge -ZD\ge 3$.
We have as above a surjective map
\[
H^0(\cO_D(-Z))\otimes H^0(\cO_D(-Z-D_s)) 
\to H^0(\cO_D(-2Z-D_s))
\]
and $H^0(\cO_X(-2Z-D_s)) \subset 
I H^0(\cO_X(-Z-D_s))+H^0(\cO_X(-2Z-D_s-D))$.
From these arguments, for $m>0$, we have 
$I_{2Z}\subset I^2+H^0(\cO_X(-2Z-mD))$. 
We denote by $H(m)$ the minimal anti-nef cycle on $X$ 
such that $H(m)\ge 2Z+mD$. 
Then $H^0(\cO_X(-2Z-mD))=H^0(\cO_X(-H(m)))$, 
and for an arbitrary $n\in \bbZ_{+}$ 
there exists $m(n)\in \bbZ_{+}$ such that $H(m(n))\ge nE$. 
Therefore, $H^0(\cO_X(-2Z-mD))\subset  I^2$ 
for sufficiently large $m$, and we obtain $I_{2Z}=I^2$.
\qed 
\end{proof} 

%%%%%%%%%%   Remark 3.24
\begin{rem}\label{r:ZD3}
Assume that $A$ is elliptic.
It follows from Proposition \ref{p:Cq} and Corollary \ref{c:ellv} that 
$q(I)=0$ if and only if $ZD\ne 0$.
By an argument similar to the proof of Theorem \ref{I2}, 
we can prove that if $ZD\ne 0$, then $I=I_Z$ is normal 
if and only if $-ZD\ge 3$.
\end{rem}

\par 
For elliptic ideals in an elliptic singularity (not strongly elliptic), 
Remark \ref{r:ZD3} cannot be applied 
because the condition $ZD \ne 0$ does not hold in general. 
Next example shows that the condition $0 < -ZD <3$ is 
not necessary for $I_Z$ being not normal.

%%%%%%  Example 3.25
\begin{ex} \label{noti}
Suppose that $p\ge 1$ be an integer.
Let $ A=k[x,y,z]/(x^2+y^3+z^{6(p+1)})$ 
and assume that  $X$ is the minimal resolution. 
Then $E$ is a chain of $p+1$ nonsingular curves 
$E_0, E_1, \dots, E_p$, where $g(E_0)=1$, $E_0^2=-1$, $g(E_i)=0$, 
$E_i^2=-2$,  $E_{i-1}E_i=1$ for $1\le i \le p$ and 
$E_iE_j =0$ if $|i-j| \ge 2$.  
It is easy to see that $A$ is elliptic and $E_0$ is 
the minimally elliptic cycle.
Furthermore, $\m$ is a $p_g$-ideal and $p_g(A)=p+1$ 
by \cite[3.1, 3.10]{OWY4}.
Since $A$ is not strongly elliptic, 
there is a non-$p_g$-ideal $I_Z$ such that $-ZE_0=0$ 
(see Theorem \ref{pg=1} and Proposition \ref{p:Cq}).
Let $W=\sum_{i=0}^p(p+1-i)E_i$. 
Then $-W\sim K_X$ and the exceptional part of 
the divisors $\di_X(x)$, $\di_X(y)$, $\di_X(z)$ 
are $3W$, $2W$, $E$, respectively.
For $1\le n \le p+1$, 
let $D_n=\gcd(nE, W):=\sum_{i=0}^p\min(n,p+1-i)E_i$. 
(Our cycle $D_n$ coincides with $C_{n-1}$ in \cite[2.6]{Ok}.) 
Then $\cO_X(-2D_n)$ is generated (cf. \cite[3.6 (4)]{Ok}) 
and $D_n^2=-n$.
Let $I_n=I_{2D_n}$. Since the cohomological cycle of 
$(D_n)^{\bot}$ is $W-D_n$, we have 
$q(I_n)=p_g(A)-n$ by Proposition \ref{p:Cq}; 
note that $-(W-D_n)\sim K_X$ on a neighborhood of 
$\supp(W-D_n)=E_0\cup\cdots\cup E_{p-n}$.
Then $I_n=(x,y,z^{2n})$. 
We have $D_nE_0=0$ for $1\le n\le p$ and $D_{p+1}E_0=E_0^2=-1$.
Therefore, it follows from Remark \ref{r:ZD3} that 
$\overline{I_{p+1}^2}\ne I_{p+1}^2$, since $-2D_{p+1}E_0=2$.
However, the condition $0<-ZE_0\le 2$ is not necessary 
for $I_Z$ being not normal.
In fact, we have  $\overline{I_{n}^2}\ne I_{n}^2$ for all $1\le n \le p+1$ 
because $xz^n\not \in I_n^2$ and $(xz^n)^2\in I_n^4$.
\end{ex}

%%%%%%%%%%%%%%%%%%%%%%%%%%%%%%%%%%%%%%%%%%%%%%%%%%%%%%%%%
%%%%%   Sedction 4
\section{The existence of strongly elliptic ideals}

Motivated by the fact that  in every two-dimensional excellent 
normal local domain which is not a rational singularity 
elliptic ideals always exist, it is natural to ask if it is also true for    
strongly elliptic ideals. 
We need some more preliminaries for proving that the answer is negative, in particular there are two-dimensional 
excellent normal local domains with no integrally closed 
$\m$-primary ideals $I$ with $\bar e_2(I)=1$. 
Assume  $(A,\m)$ is a two-dimensional excellent normal local domain 
over an  algebraically closed field.  
\par 
Let $\pi\colon X \to \Spec A $ be a resolution of singularity with exceptional set $E=\bigcup E_i$.

%%%%%   Definition 4.1
\begin{defn} \label{cycle}
Let $D\ge 0$ be a cycle on $X$
and let
\[
h(D)=\max 
\left\{h^1(\cO_{B}) \,\bigg| \,B \in \sum \bbZ E_i, \; B\ge 0, \; 
\supp (B)\subset \supp (D)\right\}
%  \defset{h^1(\cO_{B})}{B \in \sum \Z E_i, \; B\ge 0, \; 
%\supp (B)\subset \supp (D)}.
\]
We put $h^1(\cO_{B})=0$ if $B=0$.
There exists a unique minimal cycle $C\ge 0$ such that 
$h^1(\cO_C)=h(D)$ (cf. \cite[4.8]{Re}). 
We call $C$ the {\em cohomological cycle} of $D$.
The cohomological cycle of $E$ is denoted by $C_X$.
\par 
Note that $p_g(A)=h(E)$ and that 
if $A$ is Gorenstein and $\pi$ is the minimal resolution, then the 
canonical cycle  $Z_{K_X}=C_X$ (\cite[4.20]{Re}).
Clearly, the minimally elliptic cycle is the cohomological cycle of itself.
\end{defn}
%%%%%%%%%%%%%%%%%%
%%%%   Remark 4.2
\begin{rem}
(1) If $C_1$ and $C_2$ are cohomological cycles of some cycles on $X$ such that $C_1 \le C_2$ and $h^1(\cO_{C_1})<h^1(C_2)$, then $\supp(C_1) \ne \supp(C_2)$. 
\par 
(2) In general, for $q<p_g(A)$, cohomological cycles $C$ with $h^1(\cO_C)=q$ is not unique. 
For example, there exists a singularity whose minimal good resolution has two minimally elliptic cycles  (e.g., \cite{NO}). 
\end{rem}

\par 
The following result is a generalization of \cite[2.6]{OWY3}. 

%%%%%%   Proposition 4.3
\begin{prop} \label{p:CC}
Assume that $p_g(A)>0$ and 
let $D\ge 0$ be a reduced cycle on $X$.
Then the cohomological cycle $C$ of $D$ is 
the minimal cycle such that 
$H^0(X\setminus D, \cO_X(K_X))=H^0(X,\cO_X(K_X+C))$.
Therefore, 
if $g\:X'\to X$ is the blowing-up at a point in $\supp C$ and 
$E'$ the exceptional set for $g$, then the cohomological cycle $C'$ 
of $g^*D$ satisfies that $g_*^{-1}C \le C' \le g^*C-E'$  and 
$h^1(\cO_{C'})=h^1(\cO_C)$; 
we have $C'=g^*C-E'$ if $\cO_X(K_X+C)$ is 
generated at the center of the blowing-up.
\end{prop}

\begin{proof}
Let $F>0$ be an arbitrary cycle with $\supp(F)\subset D$.
By the duality, we have $h^1(\cO_F)=h^0( \mathcal{O}_F(K_X+F))$. 
From the exact sequence 
\[
0 \to \cO_X(K_X) \to \cO_X(K_X+F) \to \cO_F(K_X+F) \to 0
\]
and the Grauert-Riemenschneider vanishing theorem, we have
\begin{equation}\label{eq:h1F}
h^1(\cO_F)=\ell_A(H^0(X, \cO_X(K_X+F))/H^0(X,\cO_X(K_X))).
\end{equation}
On the other hand, we have the inclusion
\[
H^0(X, \cO_X(K_X+F)) \subset H^0(X\setminus D, \cO_X(K_X)), 
\]
where the equality holds if $F$ is sufficiently large; if the equality holds, 
% $h^1(\cO_F)=h(D)$.
%   by Okuma1124
we obtain $h^1(\cO_F)=h(D)$, because the upper bound $\ell_A(H^0(X\setminus D, \cO_X(K_X))/H^0(X,\cO_X(K_X)))$ for $h^1(\cO_F)$ depends only on $\supp(D)$.
Clearly, the minimum of such cycles $F$ exists as the maximal poles of 
rational forms in $H^0(X\setminus D, \cO_X(K_X))$.
Let $D'=g^{-1}(D)$. Since $K_{X'}+g^*C-E'=g^*(K_X+C)$, we have 
\[
H^0(X', \cO_{X'}(K_{X'}+g^*C-E')) = H^0(X'\setminus D', \cO_X(K_{X'})).
\]
Hence $C' \le g^*C-E'$. The inequality $g_*^{-1}C \le C'$ is clear.
From \eqref{eq:h1F}, we have $h^1(\cO_{C'})=h^1(\cO_C)$.
If $\cO_X(K_X+C)$  is generated at the center of the blowing-up, then $\cO_{X'}(K_{X'}+g^*C-E')$ has no fixed components, and the minimality of the cycle $g^*C-E'$ follows.
\end{proof}

%%%%%   Definition 4.4
\begin{defn}
We define a reduced cycle $Z^{\bot}$ to be the sum of the components $E_i\subset E$ such that $ZE_i=0$.
\end{defn}

\par 
From \cite[3.4]{OWY2}, we have the following

%%%%%%   Proposition 4.5
\begin{prop}\label{p:Cq}
Let $I=I_Z$ be  represented by a cycle $Z$ on $X$ and denote by $C$  the cohomological cycle of $Z^{\bot}$.   
Assume $\bar{r}(I)=2$,   then $\cO_C(-Z)\cong \cO_C$ and $h^1(\cO_C)=q(I)$.
\end{prop}

\par 
The converse of the result above is described as follows.

%%%%%%%  Proposition 4.6
\begin{prop}
If $C$ is the cohomological cycle of a cycle on $X$ 
with $h^1(\cO_C)=q>0$, 
then there exist a resolution $Y\to \Spec A$ and a cycle $Z>0$ on 
$Y$ such that $\cO_Y(-Z)$ is generated and $q(I_Z)=q(\infty I_Z)=q$. 
\end{prop}

\begin{proof}
There exists a cycle $W$ on $X$ such that $WE_i<0$ for all $E_i$ and
 $\cO_X(-W)$ is generated (cf. the proof of \cite[4.5]{OWY1}). 
Let $h\in I_W$ be a general element.
First we show that there exist a resolution $Y\to \Spec A$ 
and a cohomological cycle $D$ on $Y$ with $h^1(\mathcal{O}_D)=q$ 
such that if $Z_h$ is the exceptional part of $\di_Y(h)$, 
then  $Z_h^{\bot}=D_{red}$.
We shall obtain the resolution $Y$ from $X$ 
by taking blowing-ups appropriately as follows.
Let $H\subset X$ be an irreducible component 
of the proper transform of $\di_{\Spec A}(h)$ intersecting $C$ 
at a point $p$, and let $g\:X'\to X$ be the blowing-up at $p$.
Let $C'$ be the cohomological cycle of $g^*C$.
Then $h^1(\cO_{C'})=q$ by Proposition \ref{p:CC}.
If the intersection number $C'(g_*^{-1}H)$ is positive, 
then we take again the blowing-up at the intersection point. 
By the property of the intersection number of curves 
and Proposition \ref{p:CC}, taking blowing-ups in this manner, we 
obtain a resolution $Y\to \Spec A$ and a cohomological cycle $D$ 
which satisfy the conditions described above; 
in fact, for an exceptional prime divisor $F$ on $Y$, we have that  
$F\le Z_h^{\bot}$ if and only if $F$ does not intersect the proper 
transform of $\di_{\Spec A}(h)$.
Thus it follows from \cite[3.6]{Ok} (cf. \cite[3.4]{OWY2}) that
$\cO_Y(-nZ_h)$ is generated and $h^1(\cO_Y(-nZ_h))=q$ 
for $n\ge p_g(A)$. 
Then the cycle $Z:=p_g(A)Z_h$ satisfies the assertion.
\end{proof}

%%%%%%%   Corollary 4.7
\begin{cor} \label{existence}
There exists a strongly  elliptic ideal in $A$ if and only if there exists a cohomological cycle $C$ of a cycle on a resolution 
$Y\to \Spec A$ such that $h^1(\cO_C)=p_g(A)-1$.  
\end{cor}

%%%%%%%%   Example 4.8
\begin{ex} \label{no}
Let $C$ be a nonsingular curve of genus $g\ge 2$ and $D$ an divisor 
on $C$ with $\deg D>0$. Let $A= \bigoplus _{n\ge 0} 
H^0( C, \cO_C( nD))$ and assume that $a(A)=0$. 
Then $p_g(A)=g$ and 
$A$ has no strongly  elliptic ideals because any cycle $F$ on any 
resolution has $h^1(\cO_F)=0$ or $g$.
More precisely, if $Z E_0 = 0$, where $E_0\subset E$ denotes the 
curve of genus $g$, then $I_Z$ is a $p_g$-ideal; 
otherwise, $q(\infty I_Z)=0$.
\end{ex}

\par 
Next example shows that there are local normal Gorenstein domains  that always have  strongly elliptic ideals.
\begin{ex} \label{si}
Let $C$ be a nonsingular curve of genus $g\ge 2$ and put
\[
A=\bigoplus_{n\ge 0}H^0(\cO_C(nK_C)).
\]
Then $A$ is a normal Gorenstein ring by \cite{Wt}.
Suppose that $f\: X\to \Spec A$ is the minimal resolution.
We have
\[
p_g(A)=\sum_{n\ge 0}h^1(\cO_C(nK_C))=g+1
\]
by Pinkham's formula \cite{Pi}, $E\cong C$, 
$\cO_E(-E)\cong \cO_E(K_E)$, and $K_X=-2E$ 
(cf. \cite[4.6]{OWY1}).
Let $Y\to X$ be the blowing-up at a point $p\in E$ 
and let $E_1$ be the fiber of $p$ and $E_0$ 
the proper transform of $E$.
By Proposition \ref{p:CC}, we have $C_Y=2E_0+E_1$.
It follows from (b) of the theorem in \cite[4.8]{Re} that 
$h^1(\cO_{E_0})\le h^1(\cO_{nE_0})<p_g(A)$ for every $n\ge 1$.
Hence the cohomological cycle of $E_0$ is $E_0$ and $h^1(\cO_{E_0})=g=p_g(A)-1$. 
Therefore, $A$ has a strongly  elliptic ideal by Corollary \ref{existence}.
\par 
Next we construct a  strongly  elliptic ideal. Take a general linear form $L\in A_1\subset A$ and suppose that $\sup(\di_X(L)-E)\cap E$ consists of $\deg K_C$ points $p_1, \dots, p_{2g-2} \in E$.
Let $\phi\: X'\to X$ be the blowing-up at $\{p_1, \dots, p_{2g-2}\}$, and let $F_0$ be the proper transform of $E$ and $F_i=\pi^{-1}(p_i)$.
Let $Z=F_0+2(F_1+\cdots+F_{2g-2})$. 
Then $\cO_{X'}(-Z)$ is generated, $ZC_{X'}\ne 0$, and $ZF_0=0$.
Thus $q(I_Z)=g=p_g(A)-1$ (cf. \cite[3.4]{OWY2}).
Moreover, we have that $\ell_A(A/I_Z)=g$ by Theorem \ref{kato} and 
$I_Z=\overline{\m^2}+(L)$. Note that if $C$ is not hyperelliptic, then 
$\m^2$ is normal, because $A$ is a standard graded ring.
\end{ex}

\begin{acknowledgement}
We would like to thank Dr. J{\'a}nos Nagy and Professor Andr{\'a}s N{\'e}methi for their remark on Proposition \ref{p:CC}. 
We would also like to thank the referee for
useful suggestions that improved the presentation of the paper. Kei-ichi Watanabe would
like to thank the Department of Mathematics of the University of Genoa for the hospitality
during his stay in Genoa, where this joint work started.
\end{acknowledgement}

%%%%%%%%%%%%%%%%%%%%%%%%%%%%%%%%%%%%%%%%%%%%%%%%%%%%%%%%%%%%

%\makeaddress % To print address at the end.

\begin{thebibliography}{1}


\bibitem[1]{Ar}
M.~Artin, \emph{On isolated rational singularities of surfaces}, 
American J. of Math.  
\textbf{88} (1966),  129--136.
  
\bibitem[2]{CPR1}
A.~Corso, C.~Polini, M.~E.~Rossi, 
\emph{ Depth of associated graded rings via Hilbert coefficients of ideals},  J. Pure Appl. Algebra  \textbf{201}  (2005), no. 1-3, 126--141.     

\bibitem[3]{CPR2}
\bysame, 
%Alberto Corso, Claudia Polini, Maria Evelina Rossi, 
\emph{Bounds on the normal Hilbert coefficients},  
Proc. Amer. Math. Soc. \textbf{144}  (2016), 1919--1930.     

 
\bibitem[4]{C}
S.~D.~Cutkosky, \emph{A new characterization of rational surface singularities}, 
Invent. Math.  \textbf{102} (1990),  157--177.

\bibitem[5]{FT}
M.~Franciosi and E.~Tenni, \emph{On {C}lifford's theorem for singular
  curves}, Proc. Lond. Math. Soc. (3) \textbf{108} (2014), no.~1, 225--252.
  
\bibitem[6]{G} J.~Giraud,
\textit{Improvement of Grauert-Riemenschneider's
Theorem for a normal surface},
Ann. Inst. Fourier, Grenoble \textbf{32} (1982), 13--23.

\bibitem[7]{GN}
S.~Goto and K.~Nishida, 
\emph{The Cohen-Macaulay and Gorenstein Rees algebras Associated 
to Filtrations},  
%Memoirs of the Amer. Math. Soc. 1994, No. \textbf{526}. 
Mem. Amer. Math. Soc. \textbf{526}, Amer. Math. Soc., Providence, RI, 1994.

\bibitem[8]{GW}
S.~Goto and K.-i.~Watanabe, \emph{On graded rings. {I}}, J. Math. Soc. Japan
  \textbf{30} (1978), no.~2, 179--213.

\bibitem[9]{GS}
S.~Goto and Y.~Shimoda, 
\emph{On the {R}ees algebras of
{C}ohen-{M}acaulay local rings}, Commutative algebra ({F}airfax, {V}a.,1979), Lecture Notes in Pure and Appl. Math., vol.~68, Dekker, New York, 1982, pp.~201--231.% \MR{655805 (84a:13021)}



\bibitem[10]{Hun} 
C.~Huneke, 
\emph{Hilbert functions and symbolic powers}, 
Michigan Math.~J. \textbf{34} (1987), no~.2, 293--318. 


\bibitem[11]{It1} S.~Itoh,  
\emph{Integral closures of ideals generated by regular sequences}, 
J. of Algebra  \textbf{117},  (1988), 390--401.

\bibitem[12]{It2}
\bysame,
%Shiroh Itoh, 
\emph{Coefficients of Normal Hilbert Polynomials}, 
J. of Algebra \textbf{150} (1992), 101--117.

\bibitem[13]{It3}
\bysame,
%Shiroh Itoh, 
\emph{Hilbert coefficients of integrally closed ideals}, 
J. of Algebra \textbf{176} (1995), 638--652. 



\bibitem[14]{kato}
M.~Kato, \textit{Riemann-Roch theorem for strongly pseudoconvex manifolds of dimension 2},
Math. Ann. \textbf{222}, (1976), 243--250.

\bibitem[15]{La} 
H.~Laufer, 
\emph{On minimally elliptic singularities}, 
Amer. J. of Math. \textbf{99} (1975), 1257--1295. 

\bibitem[16]{Li} J.~Lipman,
\textit{Rational singularities with applications to 
algebraic surfaces and unique factorization},
Inst. Hautes \'Etudes Sci. Publ. Math.
\textbf{36} (1969), 195--279.


\bibitem[17]{MOR} S.~Masuti, K.~Ozeki, M.~E.~Rossi,  
\emph{A filtration of the Sally module and the first normal Hilbert coefficient}, J. of Algebra  
%DO-10.1016/j.jalgebra.2018.06.025
\textbf{571} (2021), 376--392
 (Volume in honor of Craig Huneke).

\bibitem[18]{MORT} S.~Masuti,  K.~Ozeki, M.~E.~Rossi and H.~Le Truong, \emph{On the structure of the Sally module and the second normal Hilbert coefficient},  Proc. Amer. Math. Soc.  \textbf{148} (2020), no.~7, 2757--2771.


\bibitem[19]{MSV} S.~Masuti, P.~Sarkar, J.~Verma, 
\emph{Variations on the Grothendieck-Serre Formula for Hilbert functions and their applications},
%Algebra and its Applications, \textbf{174}  (2016), 123--158.
Algebra and its applications, Springer Proc. Math.
Stat., vol. \textbf{174}, Springer, Singapore, 2016, pp. 123–158.

\bibitem[20]{NO}
A.~N\'{e}methi and T.~{Okuma}, \emph{Analytic singularities supported by a specific integral homology sphere link},
 Methods Appl. Anal.  \textbf{24} (2017), no.~2, 303--320.


\bibitem[21]{NR}
D.~G.~Northcott and D.~Rees, \emph{Reduction of ideals in local rings},
Proc.Camb. Phil. Soc.  \textbf{50} (1954), 145--158.


\bibitem[22] {Ok}T. Okuma,
 \emph{Cohomology of ideals in elliptic surface singularities}, 
 Illinois J. Math. \textbf{61} (2017), no.~3--4, 259--273.

\bibitem[23]{OWY1}
T.~{Okuma}, K.-i. {Watanabe}, and K. {Yoshida}, \emph{Good ideals and
  {$p_g$}-ideals in two-dimensional normal singularities}, Manuscripta Math.
  \textbf{150} (2016), no.~3-4, 499--520.

\bibitem[24]{OWY2}
\bysame,
%Tomohiro Okuma, Kei-ichi Watanabe  and Ken-ichi Yoshida, 
\emph{{Rees algebras and $p_g$-ideals in a two-dimensional normal local domain}}, 
 Proc. Amer. Math. Soc. \textbf{145} (2017), no. 1, 39--47.

\bibitem[25]{OWY3}
%Tomohiro Okuma, Kei-ichi Watanabe, and Ken-ichi Yoshida, 
\bysame,
\emph{{A characterization of 2-dimensional rational singularities via core of ideals}},  J. of Algebra \textbf{499} (2018), 450--468. 


\bibitem[26]{OWY4}
%Tomohiro Okuma, Kei-ichi Watanabe, and Ken-ichi Yoshida, 
\bysame,
\emph{{Normal reduction numbers for normal surface singularities 
with application to elliptic singularities of Brieskorn type}}, 
Acta Mathematica Vietnamica \textbf{44} (2019), no. 1, 87--100. 


\bibitem[27]{OWY5}
\bysame,
%Tomohiro Okuma, Kei-ichi Watanabe, and Ken-ichi Yoshida, 
\emph{The normal reduction number of two-dimensional cone-like singularities},
Proc. Amer. Math. Soc. {\bf 149} (2021), no.11, 4569--4581. 

\bibitem[28]{Pi}  H.~Pinkham, 
\emph{{Normal surface singularities with {$C\sp*$}action}}, 
Math. Ann.  \textbf{227} (1977), 183--193. 

\bibitem[29]{R} D.~Rees,  \emph{A note on analytically unramified local rings}, J. London Math. Soc. \textbf{36} (1961),
24--28. 

\bibitem[30]{Re}
M.~Reid, \emph{Chapters on algebraic surfaces}, Complex algebraic geometry,
  IAS/Park City Math. Ser., vol.~3, Amer. Math. Soc., Providence, RI, 1997,
  pp.~3--159.

\bibitem[31]{Rh}
A.~R{\"o}hr, \emph{A vanishing theorem for line bundles on resolutions of
  surface singularities}, Abh. Math. Sem. Univ. Hamburg \textbf{65} (1995),
  215--223.

\bibitem[32]{Ro} 
M.~E.~Rossi,  
\emph{A bound on the reduction number of a primary ideal}, Proc. Amer. Math. Soc. \textbf{128} (1999), % No.5
1325--1332. 

 
\bibitem[33]{RV} 
M.~E.~Rossi and G.~Valla, 
\emph{Hilbert functions of Filtered Modules}, 
Lect. Notes Unione Mat. Ital. \textbf{9},
Springer, Bologna (Italy), 2010.


\bibitem[34]{S1} J.~Sally, 
\emph{Cohen-Macaulay local ring of embedding dimension $e+d-2$}, J. of Algebra \textbf{83} (1983),  
393--408. 

\bibitem[35]{S2} 
\bysame,
%Judith Sally, 
\emph{Ideals whose Hilbert function and Hilbert polynomial agree at $n=1$}, J. of Algebra \textbf{157} (1993),   534--547. 

\bibitem[36]{T}
M.~Tomari, \emph{A {$p_g$}-formula and elliptic singularities}, Publ. Res. Inst. Math. Sci. \textbf{21} (1985), no.~2, 297--354.


\bibitem[37]{VV} 
P. ~Valabrega and G.~Valla, 
\emph{Form rings and regular sequence}, 
Nagoya Math. J. \textbf{72} (1978), 93--101. 


\bibitem[38]{W} P.~Wagreich, \emph{ Elliptic singularities of surfaces}, American J. of Math. \textbf{92} (1970),  419--454. 

\bibitem[39]{Wt}
K.-i. Watanabe, \emph{Some remarks concerning {D}emazure's construction of normal graded rings}, Nagoya Math. J. \textbf{83} (1981), 203--211.

\bibitem[40]{WY} K.-i.~Watanabe and K.~Yoshida,  \emph{Hilbert-Kunz multiplicity, McKay correspondence and
Good ideals in 2-dimensional Rational Singularities},
Manuscripta Math. \textbf{104} (2001), 275--294. 

\bibitem[41]{Y}  S.~S.~T. ~Yau, \emph{On strongly Elliptic Singularities}, 
Amer. J. of Math. \textbf{101} (1979), no~.4, 855--884. 
\end{thebibliography}
\end{document}